\documentclass[a4paper,12pt,reqno]{amsart}
\usepackage{amsfonts, color}
\usepackage{amsmath}
\usepackage{amssymb}
\usepackage[a4paper]{geometry}
\usepackage{mathrsfs}
\usepackage[colorlinks]{hyperref}
\usepackage{enumitem}
\usepackage{hyperref}
\renewcommand\eqref[1]{(\ref{#1})}

\makeatletter
\newcommand*{\mint}[1]{%
  \mint@l{#1}{}%
}
\newcommand*{\mint@l}[2]{%
  \@ifnextchar\limits{%
    \mint@l{#1}%
  }{%
    \@ifnextchar\nolimits{%
      \mint@l{#1}%
    }{%
      \@ifnextchar\displaylimits{%
        \mint@l{#1}%
      }{%
        \mint@s{#2}{#1}%
      }%
    }%
  }%
}
\newcommand*{\mint@s}[2]{%
  \@ifnextchar_{%
    \mint@sub{#1}{#2}%
  }{%
    \@ifnextchar^{%
      \mint@sup{#1}{#2}%
    }{%
      \mint@{#1}{#2}{}{}%
    }%
  }%
}
\def\mint@sub#1#2_#3{%
  \@ifnextchar^{%
    \mint@sub@sup{#1}{#2}{#3}%
  }{%
    \mint@{#1}{#2}{#3}{}%
  }%
}
\def\mint@sup#1#2^#3{%
  \@ifnextchar_{%
    \mint@sup@sub{#1}{#2}{#3}%
  }{%
    \mint@{#1}{#2}{}{#3}%
  }%
}
\def\mint@sub@sup#1#2#3^#4{%
  \mint@{#1}{#2}{#3}{#4}%
}
\def\mint@sup@sub#1#2#3_#4{%
  \mint@{#1}{#2}{#4}{#3}%
}
\newcommand*{\mint@}[4]{%
  \mathop{}%
  \mkern-\thinmuskip
  \mathchoice{%
    \mint@@{#1}{#2}{#3}{#4}%
        \displaystyle\textstyle\scriptstyle
  }{%
    \mint@@{#1}{#2}{#3}{#4}%
        \textstyle\scriptstyle\scriptstyle
  }{%
    \mint@@{#1}{#2}{#3}{#4}%
        \scriptstyle\scriptscriptstyle\scriptscriptstyle
  }{%
    \mint@@{#1}{#2}{#3}{#4}%
        \scriptscriptstyle\scriptscriptstyle\scriptscriptstyle
  }%
  \mkern-\thinmuskip
  \int#1%
  \ifx\\#3\\\else_{#3}\fi
  \ifx\\#4\\\else^{#4}\fi
}
\newcommand*{\mint@@}[7]{%
  \begingroup
    \sbox0{$#5\int\m@th$}%
    \sbox2{$#5\int_{}\m@th$}%
    \dimen2=\wd0 %
    \let\mint@limits=#1\relax
    \ifx\mint@limits\relax
      \sbox4{$#5\int_{\kern1sp}^{\kern1sp}\m@th$}%
      \ifdim\wd4>\wd2 %
        \let\mint@limits=\nolimits
      \else
        \let\mint@limits=\limits
      \fi
    \fi
    \ifx\mint@limits\displaylimits
      \ifx#5\displaystyle
        \let\mint@limits=\limits
      \fi
    \fi
    \ifx\mint@limits\limits
      \sbox0{$#7#3\m@th$}%
      \sbox2{$#7#4\m@th$}%
      \ifdim\wd0>\dimen2 %
        \dimen2=\wd0 %
      \fi
      \ifdim\wd2>\dimen2 %
        \dimen2=\wd2 %
      \fi
    \fi
    \rlap{%
      $#5%
        \vcenter{%
          \hbox to\dimen2{%
            \hss
            $#6{#2}\m@th$%
            \hss
          }%
        }%
      $%
    }%
  \endgroup
}
\numberwithin{equation}{section}
\theoremstyle{plain}
\newtheorem{thm}{Theorem}[section]

\newtheorem{cor}[thm]{Corollary}
\newtheorem{lem}[thm]{Lemma}
\newtheorem{ob}[thm]{Observation}
\theoremstyle{definition}

\newtheorem{rem}[thm]{Remark}

\newcommand{\G}{\mathbb{G}}

\newcommand{\X}{\mathbb{X}}
\def\G{\mathbb{G}}

\def\H{\mathbb{H}^{n}}
\def\X{\mathbb{X}}

\def\L{\mathcal{L}_{p}}

\def\p{2^{*}_{\beta}}

\def\gnorm{\|g\|_{L^{2}(\H,\mu)}}

\def\G{{\mathbb G}}
\def\L{\mathcal{L}}
\def\Ls{\mathcal{L}_{s}}
\def\Lss{\mathcal{L}^{s}}
\def\Rn{\mathbb{R}^{n}}
\def\H{\mathbb{H}^{n}}

\def\na{\nabla_{\mathbb{H}^{n}}}
\def\tr{\text{tr}}

\title[Log-Sobolev, Hardy and Poincar\'e inequalities on the Heisenberg group]{Logarithmic Sobolev, Hardy and Poincar\'e inequalities on the Heisenberg group}
\author[M. Chatzakou]{Marianna Chatzakou}
\address{
	Marianna Chatzakou:
	\endgraf
    Department of Mathematics: Analysis, Logic and Discrete Mathematics
    \endgraf
    Ghent University, Belgium
  	\endgraf
	{\it E-mail address} {\rm marianna.chatzakou@ugent.be}
		}
	
\author[A. Kassymov]{Aidyn Kassymov}
\address{
  Aidyn Kassymov:
  \endgraf
  Institute of Mathematics and Mathematical Modeling
  \endgraf
  28 Shevchenko str.
  \endgraf
  050010 Almaty
  \endgraf
  Kazakhstan
  \endgraf
	{\it E-mail address}  {\rm kassymov@math.kz}}

\author[M. Ruzhansky]{Michael Ruzhansky}
\address{
  Michael Ruzhansky:
  \endgraf
  Department of Mathematics: Analysis, Logic and Discrete Mathematics
  \endgraf
  Ghent University, Belgium
  \endgraf
 and
  \endgraf
  School of Mathematical Sciences
  \endgraf
  Queen Mary University of London
  \endgraf
  United Kingdom
  \endgraf
  {\it E-mail address} {\rm michael.ruzhansky@ugent.be}
  }

\begin{document}

\thanks{The authors are supported by the FWO Odysseus 1 grant G.0H94.18N: Analysis and Partial Differential Equations and by the Methusalem programme of the Ghent University Special Research Fund (BOF) (Grant number 01M01021). Marianna Chatzakou is a postdoctoral fellow of the Research Foundation – Flanders (FWO) under the postdoctoral grant No 12B1223N. Michael Ruzhansky is also supported by EPSRC grants EP/R003025/2 and EP/V005529/1, and Aidyn Kassymov by the MESRK grant AP14869275. }

\begin{abstract} 
In this paper we first prove a number of important inequalities with explicit constants in the setting of the Heisenberg group $\H$. This includes the fractional and integer Sobolev, Gagliardo-Nirenberg, (weighted) Hardy-Sobolev,  Nash inequalities, and their logarithmic versions. In the case of the first order Sobolev inequality, our constant recovers the sharp constant of Jerison and Lee. Remarkably, we also establish the analogue of the Gross inequality with a semi-probability measure on $\H$ that allows- as it happens in the Euclidean setting- an extension to infinite dimensions, and particularly can be regarded as an inequality on the infinite dimensional $\mathbb{H}^{\infty}$. Finally, we prove the so called generalised Poincar\'e inequality on $\H$ both with respect to the aforementioned semi-probability measure and the Haar measure on $\H$, also with explicit constants.
%In this paper we prove a number of fractional and integer logarithmic inequalities on Heisenberg group: log-Sobolev, log-Gagliardo-Nirenberg, weighted log-Hardy-Sobolev, log-Hardy and Gross inequalities on Heisenberg group. We also obtain  of  Nash inequality for Heisenberg group and its application to heat equation. The fractional Sobolev inequality on Heisenberg group is playing a key role in the proofs of above results.
\end{abstract}

\maketitle

\tableofcontents

\section{Introduction and main results}
One of the aims of this work is to obtain the fractional, logarithmic and in some cases weighted versions of a family of well-studied inequalities such as the Sobolev, Hardy-Sobolev, Hardy, and Gagliardo-Nirenberg inequalities  in the setting of the Heisenberg group $\mathbb{H}^n$. We put an emphasis on establishing there inequalities with explicit constants. It is true that even though a lot of work has been done in order to prove and extend these inequalities in the Euclidean setting, regarding their consideration in the sub-Riemannian setting there is still a lot to be done. Exceptional to this, we shall refer to the important works of Frank and Lieb \cite{FL12} and Roncal and Thangavelu \cite{RT16} where the authors prove several sharp inequalities in the setting of the Heisenberg group $\H$. 

Additionally, in the last section we prove the so-called generalised Poincar\'e inequality (see \cite{Bec98}) in the regarded setting, with respect to two different and essential for our considerations measures on $\H$. The versions of the Poincar\'e inequalities proven are both global; integration is regarded on the whole of $\H$, and local; integration is regarded on balls in $\H$.

One of the main contributions of the current work is that the involved constants in the aforementioned inequalities are explicit. In this respect, the current work can be served as a special case of the works \cite{CKR21c,CKR21b, CKR21a} in the sense that in these works the authors established similar type inequalities for general or graded/stratified Lie groups including just representations of the appearing constants, but also an extension of the latter works since the current work contains several Poincar\'e types inequalities. The importance of knowing explicitly the appearing constants is self-evident in view of the applications of these inequalities in PDEs, variational calculus, differential geometry, and other branches of mathematics as discussed in more details in the sequel. Additionally to that, having control over the constants is necessary in view of the possibly of extending the consideration of specific inequalities to infinite-dimensional objects like, in this case, the infinite dimensional Heisenberg groups $\mathbb{H}^{\infty}$, see e.g. the works \cite{DM08} and  \cite{BGM13} on the setting $\mathbb{H}^{\infty}$, when equipped with a suitable probability measure. Details of this consideration are given below in \ref{iii}.

On the other hand, it is known, see e.g \cite{BZ05}, \cite{ABD07}, \cite{DV12}, that there is a strong link between the logarithmic-Sobolev inequalities and the Poincar\'e inequalities. Here we establish several versions of the Poincar\'e inequalities in the setting of the Heisenberg group $\H$, where either the Haar measure is involved, or a semi-probability measure that reads as the natural extension of the Gaussian measure on $\mathbb{R}^n$. This appears in the analogue of the logarithmic Sobolev inequality in the setting of $\H$ that is proven in the current work. 

%Regarding the main motivation of the current work, let us mention that, it is indeed a natural question whether the logarithmic form of the above inequalities still allows the inequalities to hold. However, especially in view of the work \cite{CKR21c} which will be discussed in detail in the sequel, one could anticipate the affirmative reply to this question for almost all the inequalities studied here excluding those involving  the fractional Laplacian or the modified fractional power of it $\mathcal{L}_s$ that will be introduced very shortly. 
Hence the stimulus behind this work is mainly based on the following elements:
\begin{enumerate} [label=\roman*.]
    \item \label{a} Logarithmic, or more generally coercive inequalities, are strongly linked  other inequalities and  properties of the related semigroups; see e.g. \cite{AB00, GZ03,BH99, Car91,HZ10};
\item   \label{b} Particularly, the log-Sobolev inequality has links to measuring uncertainty, c.f. \cite{Bec95,Bec00}, and to blow-up results on certain PDEs even in the quite general setting of homogeneous Lie groups, c.f. \cite{KRS20}; Nash inequality considered in the setting of Markov semigroups can be used to proving estimates for the related heat kernel; Hardy inequality can be used to  proving stability of the relativistic matter \cite{FLS08}; and Gagliardo-Nirenberg inequalities (also known as ``interpolation theorem'' in nonlinear problems) can be used to proving estimates for the study of nonlinear evolution equations, c.f. \cite{FNQ18};
\item \label{iii}  While \ref{a} and \ref{b} explain the general motivation behind our analysis, here we obtain explicit expression of the appearing constants, and having good control of these constants allows one to pass to infinite dimensions.  In particular, this means that we can still consider the Gross-type ``semi-Gaussian'' inequality with the Gaussian measure on the first stratum of the group, see Theorem \ref{semi-g}, on an object like $\mathbb{H}^{\infty}$ with the infinite-dimensional first stratum and a $1$-dimensional center. This relies on the  appearing normalisation constant $\gamma$, see Remark \ref{mainrem}, that essentially allows one to interpret the measure as the probability measure on the whole $\mathbb{H}^{\infty}$ as it happens in the Euclidean setting. Hence it becomes consequential to study properties of the analogues of the Ornstein–Uhlenbeck operator on $\mathbb{H}^{\infty}$, see e.g. \cite{Pic10}. The last realisation unlocks new investigation quests, eventually including the direction of stochastic analysis on $\mathbb{H}^{\infty}$. 
\item  \label{iv} Sobolev and Poincar\'e iequalities are closely related \cite{LR04}. The Poincar\'e inequality is the main protagonist in the domain of coercive inequalities and isoperimetry; see \cite{BZ05} for a general overview. In the case of the global Poincar\'e inequalities with respect to a probability measure very little is known even in the simple setting of the Heisenberg group $\H$, see e.g. \cite{CFZ21}, \cite{DZ22}, \cite{HZ10}. On the other hand, Poincar\'e inequalities on a bounded domain in the setting of nilpotent Lie groups are known only with respect to the horizontal gradient on the group, see \cite{Jer86}.
\end{enumerate}

Next, we give a short presentation of the main results of this work, sometimes combined with some additional comments about related works and approaches. Before doing so, let us note that in the sequel $\H $ stands for the Heisenberg group, and $Q=2n+2$ stands for the homogeneous dimension of $\H$. We also assume that for the parameter $s$ we have $s \in (0,1]$.\footnote{While the definition of the modified sub-Laplacian $\mathcal{L}_s$ is given for $s\in (0,1)$, here we allow the parameter $s$ to also take the extreme value $1$ since in this case $\mathcal{L}^s$ boils down to the sub-Laplacian itself, and this consideration makes sense in our setting.} 

\begin{itemize}
\item{\textbf{The (fractional) (log-)Sobolev inequality on $\H$}:}  The works \cite{Bec12, KRS20, KS20} contain functional inequalities, in some cases of Sobolev type, that involve powers of the sub-Laplacian. Our version of the fractional Sobolev inequality as appears in Theorem \ref{thm1} involves either the modified fractional operator $\mathcal{L}_s$ given by \eqref{Ls.Thang} (introduced by Roncal and Thangavelu in \cite{RT16} that in the case where $s=1$ boils down to the sub-Laplacian on $\H$) or fractional powers of the sub-Laplacian $\Lss$ and the first reads as follows:
    \begin{equation}\label{into1}
        \|f\|^{2}_{L^{\frac{2Q}{Q-2s}}(\H)}\leq C_{B,s}\langle \Ls f,f\rangle_{L^{2}(\H)},
    \end{equation}
where $C_{B,s}$ is given in \eqref{bestfrac}. The fractional Sobolev inequality that involves $\Lss$ is the inequality \eqref{2b} in Theorem \ref{thm1}. Inequality \eqref{into1} essentially extends the sharp inequality of Jerison and Lee in \cite{JL88} in the sense that when $s=1$, i.e. when $\mathcal{L}_s$ is simply the sub-Laplacian on $\H$, inequality \eqref{into1} coincides with the latter inequality, see also Remark \ref{remJL}. It is important to note that when combining the latter with the logarithmic H\"older inequality as in Lemma \ref{holder}, then the resulting inequality gives as a special case for $s=1$ the following so called Log-Sobolev inequality on $\H$:
  \begin{equation}\label{LogSobolevint.intro}
\int_{\H}\frac{|u|^{2}}{\|u\|^{2}_{L^{2}(\H)}}\log\left(\frac{|u|^{2}}{\|u\|^{2}_{L^{2}(\H)}}\right) dx \leq \frac{Q}{2}\log\left(C_{B,1}\|\nabla_{\H}u\|^{2}_{L^{2}(\H)}\right),
\end{equation}
where the constant  $C_{B,1}$ is given explicitly. A generalisation of the log-Sobolev inequality is its weighted version and is shown in Theorem \ref{wwsobingrthm}. 
  %  \begin{equation*}
   %     C_{B,s}=\frac{\pi^{-s}(n!)^{\frac{s}{n+1}}2^{-2s}\Gamma^{2}\left(\frac{Q-2s}{4}\right)}{\Gamma^{2}\left(\frac{Q+2s}{4}\right)}.
    %\end{equation*}
    %Also, $s\in(0,1)$, we have
    %\begin{equation*}
    %\begin{split}
    %    \|f\|^{2}_{L^{\frac{2Q}{Q-2s}}}\leq\frac{\pi^{-s}(n!)^{\frac{s}{n+1}}2^{-2s}\Gamma^{2}\left(\frac{Q-2s}{4}\right)}{\Gamma^{2}\left(\frac{Q+2s}{4}\right)}\|U_{s}\|_{\op}\langle \Ls f,f\rangle_{L^{2}(\H)},  
    %    \end{split}
    %\end{equation*}
  %  and
   % \begin{equation}\label{2b}
    %\begin{split}
     %    \|u\|^{2}_{L^{\frac{2Q}{Q-2s}}}&\leq \frac{\pi^{-s}(n!)^{\frac{s}{n+1}}2^{-\frac{2s}{n+1}}\Gamma^{2}\left(\frac{Q-2s}{4}\right)}{\Gamma^{2}\left(\frac{Q+2s}%{4}\right)}\|V_{s}\|_{\text{OP}}\langle\Lss u,u\rangle_{L^{2}(\H)},
    %\end{split}
  %  %\end{equation}
   
\item{\textbf{The (log-)fractional Gagliardo-Nirenberg inequality on $\H$}:} As one sees in the seminal work of Brezis and Mironescu \cite{BM18} on the history of Gagliardo-Nirenberg inequality, not much attention was paid until recently to the determination of the appearing constant even in the simple case of $\mathbb{R}^n$. In 2020 in the work \cite{RTY20} the authors prove the Gagliardo-Nirenberg inequality on graded groups with a constant that cannot (apart from the trivial cases of $\mathbb{R}^n$) be explicitly computed. Here, in Theorem \ref{GN4},  in the case of $\H$ we obtain:
\begin{equation}\label{into.GN}
 \int_{\H}|u(x)|^{q}dx\leq C_{GN,s_1,s_2}\langle\mathcal{L}_{s_{1}}u,u\rangle_{L^{2}(\H)}^{\frac{Q(q-2)-2s_{2}q}{4(s_{1}-s_{2})}}\langle\mathcal{L}_{s_{2}}u,u\rangle_{L^{2}(\H)}^{\frac{s_{1}(q-2)-Q(q-2)}{4(s_{1}-s_{2})}},
\end{equation}
    where $C_{GN,s_1,s_2}=C_{B,s_1}^{\frac{qa}{2}} C_{B,s_2}^{\frac{q(1-a)}{2}}$ for $C_{B,s_i}$, $i=1,2$, explicitly given in \eqref{bestfrac}, whenever $1>s_1>s_2 \geq 0$, and $\frac{2Q}{Q-2s_2}\leq q \leq \frac{2Q}{Q-2s_1}$. In principle, the Gagliardo-Nirenberg inequality on $\H$ in its general form is not sharp. However, in the case where $s_1=1$ and $s_2=0$ inequality \eqref{into.GN} is sharp, see Remark \ref{rem.GN}. Logarithmic versions of the general form of \eqref{into.GN} are given in Theorem \ref{thmlogGN}.
\item {\textbf{(log-)fractional Hardy--Sobolev inequality on $\H$:}} The (log-)Sobolev type inequalities with weights are sometimes called (log-)Hardy inequalities. The fractional Hardy inequality on $\H$ is shown in \cite{RT16} by Roncal and Thangavelu. Here in Theorem \ref{FHS} we prove the log-Hardy-Sobolev inequality on $\H$, which is for $s \in (0,1)$, $2^{*}_{\beta}=\frac{2(Q-\beta)}{Q-2s}$, and $0< \beta <2s$ the following inequality   
      \begin{equation*}
  \begin{split}
       &\int_{\H}\frac{|x|^{-\frac{2\beta}{2^{*}_{\beta}}}|u(x)|^{2}}{\||\cdot|^{-\frac{2\beta}{2^{*}_{\beta}}}u\|^{2}_{L^{2}(\H)}}\log\left(\frac{|x|^{-\frac{2\beta}{2^{*}_{\beta}}}|u(x)|^{2}}{\||\cdot|^{-\frac{2\beta}{2^{*}_{\beta}}}u\|^{2}_{L^{2}(\H)}}\right) dx\\&
       \leq \frac{Q-\beta}{2s-\beta}\int_{\H}\frac{|u|^{2}}{|x|^{2s-\beta}}dx\log\left(\left(C^{\frac{\beta}{2s}}_{BH,s}(C_{B,s}\|U_{s}\|_{\text{op}})^{\frac{n+1}{n+1-s}\frac{2s-\beta}{2s}}\right)^{\frac{2}{2^{*}_{\beta}}}\frac{\langle \Lss u,u\rangle_{L^{2}(\H)}}{\||\cdot|^{-\frac{2\beta}{2^{*}_{\beta}}}u\|^{2}_{L^{2}(\H)}}\right)\,,
   \end{split}
\end{equation*}
  where $\|U_{s}\|_{\text{op}}$ stands for the norm of the operator $U_s$ defined in \eqref{Us}. The latter inequality, see the discussion given after Theorem \ref{FHS}, gives under suitable choices of $\beta$ both the Hardy and the Sobolev inequalities. The fractional log-Hardy inequality on $\H$ is given in Theorem \ref{THM:weHarnon}.
%\item{ \textbf{log-Hardy inequalities on $\H$}:} 

%Let $\H$ be a Heisenberg group with homogeneous dimension $Q$ and let $s\in(0,1)$ such that $0<\beta<2s$. Then for any $u\in H^{s}(\H)$, we have
   
 %     \begin{equation*}
 % \begin{split}
  %     \int_{\H}\frac{|x|^{-\frac{2\beta}{2^{*}_{\beta}}}|u(x)|^{2}}{\||\cdot|^{-\frac{2\beta}{2^{*}_{\beta}}}u\|^{2}_{L^{2}(\H)}}&\log\left(\frac{|x|^{-\frac{2\beta}{2^{*}_{\beta}}}|u(x)|^{2}}{\||\cdot|^{-\frac{2\beta}{2^{*}_{\beta}}}u\|^{2}_{L^{2}(\H)}}\right) dx\\&
  %     \leq \frac{Q-\beta}{2s-\beta}\log\left(\left(C^{\frac{\beta}{2s}}_{BH,s}(C_{B,s}\|U_{s}\|_{\text{OP}})^{\frac{n+1}{n+1-s}\frac{2s-\beta}{2s}}\right)^{\frac{2}{2^{*}_{\beta}}}\frac{\langle \Lss u,u\rangle_{L^{2}(\H)}}{\||\cdot|^{-\frac{2\beta}{2^{*}_{\beta}}}u\|^{2}_{L^{2}(\H)}}\right).
   %\end{split}
%\end{equation*}
 % where $u\in H^{s}(\mathbb{H}^{n})$ and $2^{*}_{\beta}=\frac{2(Q-\beta)}{Q-2s}$;
\item{\textbf{The Nash inequality and its application to heat equation on $\H$}:} Nash inequality was shown by Nash \cite{Nas58} in the Euclidean setting. The best-appearing constant was determined by Carlen and Loss in \cite{CL} a few years later. In their monograph Varopoulos, Salof-Coste and Coulhon \cite{VSCC93} show Nash inequality in the form \begin{equation}
\label{intro.VSDD}\|u\|^{2+\frac{4}{Q}}_{L^2(\H)}\leq C \|u\|^{\frac{4}{Q}}_{L^1(\H)}\|\nabla_{\H}u\|^{2}_{L^2(\H)}\,,
\end{equation} for some $C>0$. Here we find an explicit formula for $C$ in \eqref{intro.VSDD} which, as it happens on $\mathbb{R}^n$ depends only on the dimension $n$ of the first stratum, and consequently only on $Q$, and generalises \eqref{intro.VSDD} in the following sense:

\begin{equation}\label{intro.Nash}
 \|u\|_{L^2(\H)}^{2+\frac{4s}{Q}} \leq C_{B,s}\|U\|_{\rm op} \|u\|_{L^1(\H)}^{\frac{4s}{Q}}\langle\Lss u,u\rangle_{L^{2}(\H)}\,,
 \end{equation}
 see Theorem \ref{THM:Nash}. The modified fractional version of Nash inequality is given in \eqref{Nash.g}. As before, the inequality \eqref{intro.Nash} for $s=1$,  boils down to the inequality \eqref{intro.VSDD} with an explicit constant. An application of Nash inequality for the time decay of the solution to the heat equation with respect to the sub-Laplacian $\mathcal{L}$ on $\H$ is given in Corollary \ref{cor:par}. \item{\textbf{The Gross-type log-Sobolev inequality on $\H$}:}
 Gross in \cite{Gro75} proved the logarithmic Sobolev inequality on $\mathbb{R}^n$ with respect to the Gaussian measure $d\mu(x)=(2\pi)^{-\frac{n}{2}}e^{-\frac{|x|^2}{2}}dx$ of the form:
 \begin{equation}\label{EQ:Gross}
    \int_{\mathbb{R}^{n}}|u(x)|^{2}\log  \left(\frac{|u(x)|}{\|u\|_{L^{2}(\mathbb{R}^{n},\mu)}}\right)d\mu(x)
    \leq  \int_{\mathbb{R}^{n}}|\nabla u(x)|^{2}d\mu(x)\,,
\end{equation}
with the idea of extending it to infinite dimensions. Indeed, the appearing normalisation constant $(2\pi)^{-n}$ allows for such a consideration. Gross's inequality is linked to  ultracontractivity and hypercontractivity properties of the corresponding Markovian semigroups, see e.g. \cite{Ad79, Tos97}, uncertainty principles \cite{Bec95} and Poincar\'e inequalities \cite{Bec89}.
In Theorem \ref{semi-g}  we prove the following analogue of Gross's inequality on $\H$ which we shall call the ``semi-Gaussian'' log-Sobolev (or Gross type) inequality on $\H$ since the appearing measure is Gaussian on the first stratum. The latter inequality is of the form 
\begin{equation}\label{intro:semi-g}
    \int_{\H}|g|^2\log|g|\,d\mu \leq \int_{\H} |\nabla_{\H}g|^2\,d\mu\,,
\end{equation}
for any $g$ such that $\|g\|_{L^2(\H,\mu)}=1,$
 where $\mu=\mu_1 \otimes \mu_2$, and $\mu_1$ is the Gaussian measure on $\mathbb{R}^{2n}$ given by $d\mu_1=\gamma e^{-\frac{|\xi|^2}{2}}d\xi$, for $\xi\in \mathbb{R}^{2n}$, and $|\xi|$ being the Euclidean norm of $\xi$, where 
 \begin{equation*}
     \gamma:=n!\left(\frac{(n+1)e^{\frac{2n}{n+1}{-1}}}{2\pi n^{2}}\right)^{n+1}\,,
 \end{equation*}
 where  $\mu_2$ is the Lebesgue measure $d\tau$ on $\mathbb{R}$, where we have denoted by $x=(\xi,\tau) \in \mathbb{R}^{2n} \times \mathbb{R}$ an element in $\H$. As it happens on $\mathbb{R}^n$, the latter inequality is actually equivalent to \eqref{LogSobolevint.intro}; see Theorem \ref{equiv.thm}. Let us underline that $\gamma$ for large $n$ (and so also for large $Q$) is uniformly bounded, see Observation \ref{mainrem}, and most importantly, in this limiting case, the appearing probability measure on the first stratum of $\H$ becomes identical to the Gaussian measure on $\mathbb{R}^{2n}$. Hence, the idea of passing to infinite dimensions is justified in the case of $\H$ by this expression.
 \item{\textbf{Poincar\'e inequalities on $\H$}:} Regarding Poincar\'e inequalities on the whole of $\H$, in Theorem \ref{thm.poi.H}, we prove the generalised Poincar\'e inequality, see e.g. \cite{Bec89}, with respect to the semi-Gaussian measure $\mu$ as in \eqref{intro:semi-g}. The latter reads as follows
 \[
\int_{\H}|g|^2\,d\mu-\left(\int_{\H}|g|^p\,d\mu \right)^{\frac{2}{p}}\leq  \frac{2(2-p)}{p} \int_{\H}|\nabla_{\H}g|^2\,d\mu\,,
 \]
 for $p<2$. Later on, in Corollary \ref{cor.jer}, we deal with the Poincar\'e inequality on bounded domains in $\H$, and in particular we give an upper bound when $p=2$ for the term 
 \[
 \int_{B_r(y)}|g-g_{B_r(y)}|^p\,dx 
 \]
 appearing in Jerison's Poincar\'e inequality \cite{Jer86} in terms of the infimum over $s\in (0,1]$ of the quantity 
 \[
 C_{B,s}\frac{Q}{2s} \langle \mathcal{L}_sg,g \rangle_{L^2(B_r(y))}\,,
 \]
 where $B_r(y)$ is the ball of radius $r$ with center $y$ with respect to the homogeneous quasi-norm on $\H$ which satisfies $\text{Vol}(B_r(y))\leq 1$, and where $C_{B,s}$ depends only on $s$ and is given in \eqref{bestfrac}.
\end{itemize}

Actually, Bakry-Emery conditions give the conditions on (the full) probability measures under which logarithmic inequalities are satisfied; see \cite{GZ03}. Such an approach is not part of the current work and will be addressed in the sequel. In this spirit let us mention an important result by Hebisch and Zegarlinski in \cite{HZ10} that in our case can be viewed as follows:  there is no probability measure of the form $e^{-c\rho^2}dx$ with $\rho$ being a smooth homogeneous quasi-norm on $\H$ that gives rise to a Gross-inequality of the form  \eqref{intro:semi-g}.

There will be a  subsequent work that initiates from the current one, as mentioned earlier in \ref{iii}, with the study of the properties of the corresponding Ornstein–Uhlenbeck operator on $\H$ and on $\mathbb{H}^{\infty}$. Recall that the Gross inequality \eqref{EQ:Gross} on $\mathbb{R}^n$ is corresponding essentially to a reformulation of the hypercontractivity of the associated Markovian semigroup; see e.g. \cite{BE85}.

\section{Preliminaries}
In  this  section,  we  briefly  recall  definitions and  main  properties  of  the Heisenberg group $\H$ as well as some notions and ideas from \cite{RT16} and \cite{FL12} that are useful for the subsequent analysis. 
Let us note that the comprehensive study on stratified groups, that include $\H$ as a special case, has been initiated in the works of Folland and Stein \cite{FS}. However, in our exposition below we chose to follow a more recent presentation of  the open access book \cite{FR16}. 

The Heisenberg group $\H$ can be viewed via the identification of manifolds $\mathbb{H}^{n}:=(\mathbb{C}^{n}\times \mathbb{R}, \circ),\,n\in\mathbb N$, under the following group law
\begin{equation*}
\begin{split}
(\xi,\tau)\circ(\tilde{\xi},\tilde{\tau})=\left(\xi+\tilde{\xi},\tau+\tilde{\tau}+\frac{1}{2}\text{Im}(\xi\cdot\tilde{\xi})\right),
\end{split}
\end{equation*}
 where $\xi,\tilde{\xi}\in \mathbb{C}^{n}$ and $\tau, \tilde{\tau}\in \mathbb{R}$.
The family of dilations for $z=(\xi,\tau)\in\H$ has the following form
\begin{equation*}
\delta_{\lambda}(z):=(\lambda \xi, \lambda^{2}\tau),\,\,\,\,\, \lambda>0,\,\,\,\,\xi=\xi'+\text{i}\overline{\xi},
\end{equation*}
yielding also the homogeneous dimension $Q$  of $\mathbb{H}^{n}$ which is computed as $Q=2n+2$. By definition, the topological dimension is $2n+1$. The Kaplan norm on $\H$ is for $z=(\xi,\tau)\in \H$ given by
\begin{equation}
    |z|=\left(|\xi|^{4}+16\tau^{2}\right)^{\frac{1}{4}}\,.
\end{equation}

The Lie algebra $\mathfrak{h}$ of the left-invariant vector fields on the Heisenberg group $\mathbb{H}^{n}$ is spanned by the vector fields
$$
X_{i}=\frac{\partial}{\partial\xi'_{i}}+\frac{1}{2}\overline{\xi}_{i}\frac{\partial}{\partial \tau}, \,\,\, Y_{i}=\frac{\partial}{\partial \overline{\xi}_{i}}-\frac{1}{2}\xi'_{i}\frac{\partial}{\partial \tau},\,\,\xi_{i}=\xi'_{i}+\text{i}\overline{\xi}_{i},
$$
for $i=1,\ldots,n,$ having as the (only) non-zero commutator the vector field
$$T=[X_{i},Y_{i}]=-\partial_{\tau},\,\,\,\,\,i=1,\ldots,n.$$
By the general theory, the horizontal gradient on $\mathbb{H}^{n}$ is then given by
$$\nabla_{\H}:=(X_{1},\ldots,X_{n},Y_{1},\ldots,Y_{n})\,,$$
while the positive sub-Laplacian on $\H$ is of the form
$$\L:=-\na^{*}\cdot\na=-\sum_{i=1}^{n}\left(X_{i}^{2}+Y_{i}^{2}\right).$$

Each $\lambda \in \mathbb{R}^{*}=\mathbb{R}\setminus \{0\}$ gives rise to an irreducible unitary representation $\pi_{\lambda}$ of $\H$ that is realised on $L^{2}(\mathbb{R}^{n})$ via the action 
$$\pi_{\lambda} (\xi,\tau)\varphi(t)=e^{\text{i}\lambda\tau}e^{\text{i}(\xi'\cdot t+\frac{1}{2}\xi'\cdot \overline{\xi})}\varphi(t+\overline{\xi}),$$
where $\xi=\xi'+\text{i}\overline{\xi}$.

The group Fourier  transform $\widehat{f}(\lambda)$ of a function $f \in L^1(\H)$ is the operator-valued function defined,
for each $\lambda \in \mathbb{R}^{*}$, by 
\begin{equation}
   \widehat{f}(\lambda):=\pi_{\lambda}(f)=\int_{\H}f(\xi,\tau)\pi_{\lambda}(\xi,\tau)^{*}d\xi d\tau. 
\end{equation}
The Parseval identity on $\H$ reads as follows
\begin{equation}\label{parseval}
    \int_{\H}f(x)\overline{g(x)}dx=\int_{\H}f(\xi,\tau)\overline{g(\xi,\tau)}d\xi d\tau=\frac{2^{n-1}}{\pi^{n+1}}\int_{\mathbb{R}^{*}}\text{tr}(\widehat{f}(\lambda)\widehat{g}(\lambda)^{*})|\lambda|^{n}d\lambda,
\end{equation}
where $\tr(\widehat{f}(\lambda)\widehat{g}(\lambda)^{*})$ stands for the trace of the operator $\widehat{f}(\lambda)\widehat{g}(\lambda)^{*}$. 
By taking $f=g$ in \eqref{parseval}, the latter boils down to the Plancherel identity on $\H$, i.e, we have
\begin{equation*}
    \int_{\H}|f(x)|^{2}dx=\int_{\H}|f(\xi,\tau)|^{2}d\xi d\tau=\frac{2^{n-1}}{\pi^{n+1}}\int_{\mathbb{R}^{*}}\|\widehat{f}(\lambda)\|^{2}_{\text{HS}}|\lambda|^{n}d\lambda,
\end{equation*}
where $\|\hat{f}(\lambda)\|_{\text{HS}}=\tr(\widehat{f}(\lambda)\widehat{f}(\lambda)^{*})^{\frac{1}{2}}$ is the Hilbert-Schmidt norm of $\widehat{f}(\lambda)$. 

Recall that that group Fourier transform of the sub-Laplacian $\mathcal{L}$ on $\H$ is the harmonic oscillator on $\mathbb{R}^n$;  that is we have
\begin{equation*}
   \pi_{\lambda}(\L)=\sigma_{\L}(\lambda)=|\lambda|\left(-\sum_{k=1}^{n}\partial^{2}_{w_{j}}+\sum_{k=1}^{n}w_{j}^{2}\right)\,,\quad w \in \mathbb{R}^n\,.
\end{equation*}
Next, we recall the spectral decomposition of the fractional powers of the sub-Laplacian as in \cite{RT16}. To this end, let us introduce the following notation: 
we shall denote by $f^{\lambda}$ the inverse Fourier transform of $f$ in the central variable $\tau$, that is
\begin{equation*}
    f^{\lambda}(\xi):=\int_{-\infty}^{+\infty}f(\xi,\tau)e^{\text{i}\lambda \tau}d\tau\,.
\end{equation*}
The following functions on $\mathbb{R}^n$ are called in \cite{RT16} the scaled Laguerre functions of type $(n - 1)$:
\begin{equation*}
    \varphi_{k}^{\lambda}(\xi)=L^{n-1}_{k}\left(\frac{|\lambda||\xi|^{2}}{2}\right)e^{-\frac{|\lambda||\xi|^{2}}{4}},
\end{equation*}
where $L_{k}^{n-1}$ is the Laguerre polynomial of type $(n-1)$. 
Using the above notation, the spectral decomposition of $\mathcal{L}$ was described in \cite{RT16} to be given by
\begin{equation*}
    \L u(\xi,\tau)=(2\pi)^{n-1}\int_{-\infty}^{+\infty}\left(\sum_{k=0}^{\infty}(2k+n)|\lambda|f^{\lambda}\ast_{\lambda}\varphi_{k}^{\lambda}(\xi)\right)e^{-i\lambda\tau}|\lambda|^{n}d\lambda,
\end{equation*}
where $*_{\lambda}$ is called the $\lambda$-twisted convolution and is defined via $f^{\lambda} \ast_{\lambda} g^{\lambda}:=(f \ast g)^{\lambda}$. Hence, under the above notation, a natural way to extend the latter to define fractional powers of the sub-Laplacian is via the following formula:
\begin{equation*}
    \Lss u(\xi,\tau)=(2\pi)^{n-1}\int_{-\infty}^{+\infty}\left(\sum_{k=0}^{\infty}((2k+n)|\lambda|)^{s}f^{\lambda}\ast_{\lambda}\varphi_{k}^{\lambda}(\xi)\right)e^{-i\lambda\tau}|\lambda|^{n}d\lambda.
\end{equation*}
By $\Ls$, $s\in (0,1)$, Roncal and Thangavelu in \cite{RT16} denoted the operator that they call the modified fractional power (of $\mathcal{L}$):
\begin{equation}\label{Ls.Thang}
    \Ls:=|2T|^{s}\frac{\Gamma\left(\frac{\L}{2|T|}+\frac{1+s}{2}\right)}{\Gamma\left(\frac{\L}{2|T|}+\frac{1-s}{2}\right)},
\end{equation}
where $T=-\frac{\partial}{\partial \tau}$. The authors also observe that $\Ls$ corresponds to the following spectral multiplier: 
\begin{equation*}
    (2|\lambda|)^{s}\frac{\Gamma\left(\frac{2k+n}{2}+\frac{1+s}{2}\right)}{\Gamma\left(\frac{2k+n}{2}+\frac{1-s}{2}\right)},\,\,\,\,k\in\mathbb{N},
\end{equation*}
with the group Fourier transform of the operator $\Ls$ denoted by 
\begin{equation*}
    \sigma_{\Ls}(\lambda)=\pi_{\lambda}(\Ls).
\end{equation*}
In \cite{CH89} the authors found the fundamental solution of $\Ls$. In particular, we have 
\begin{equation*}
    (\Ls^{-1}\delta_{0})(x)=a_{s}|x|^{-Q+2s},\,\,\,\,\,a_{s}=\frac{2^{n+1-3s}\Gamma^{2}\left(\frac{Q-2s}{4}\right)}{\pi^{n+1}\Gamma(s)},
\end{equation*}
where $\delta_{0}$ denotes a Dirac delta at the point $0$. It is easy to see that $\mathcal{L}_{1}=\mathcal{L}$, hence by the above the explicit representation of the fundamental solution of $\L$ is given by $\frac{2^{n-2}\Gamma^{2}\left(\frac{n}{2}\right)}{\pi^{n+1}}|x|^{-Q+2}$; see also Folland and Stein \cite{FS74}.

Since for our analysis, we are referring to several results from the work \cite{FL12} of Frank and Lieb, let us see how the above operators become under the group law on $\H$ that was adopted there. We have in \cite{FL12}, 

\begin{equation*}
\begin{split}
(\xi,\tau) \circ_{1} (\tilde{\xi},\tilde{\tau})=(\xi+\tilde{\xi},\tau+\tilde{\tau}+2\text{Im}(\xi\cdot\tilde{\xi})),\,\,(\xi,\tau)\in\mathbb{C}^{n}\times\mathbb{R}.
\end{split}
\end{equation*}
Consequently, the vector fields (with $\mathbb{C}^{n}\ni\xi=\xi'+\text{i}\overline{\xi}$) take the form
$$
\widetilde{X}_{i}=\frac{\partial}{\partial\xi'_{i}}+2\overline{\xi}_{i}\frac{\partial}{\partial \tau}, \,\,\, \widetilde{Y}_{i}=\frac{\partial}{\partial\overline{\xi}_{i}}-2\xi'_{i}\frac{\partial}{\partial \tau},\,\,\,\,i=1,\ldots,n,
$$
while the Kaplan distance of $x$, in this case denoted by $|x|_1$, takes the following form
\begin{equation}\label{fln}
    |x|^{4}_{1}:=|\xi|^{4}+\tau^{2},\,\,\,\,\,\,x=(\xi,\tau)\in\H.
\end{equation}
Hence, for the sub-Laplacian, now denoted by $\tilde{\L}$, we have
$$\tilde{\mathcal{L}}=-\frac{1}{4}\sum_{i=1}^{n}\left(\tilde{X}_{i}^{2}+\tilde{Y}_{i}^{2}\right)\,.$$
The analogue of the modified fractional operator $\mathcal{L}_{s}$ with respect to the sub-Laplacian is given by
\begin{equation*}
    \widetilde{\mathcal{L}}_{s}:=|2\widetilde{T}|^{s}\frac{\Gamma\left(\frac{\widetilde{\L}}{|2\widetilde{T}|}+\frac{1+s}{2}\right)}{\Gamma\left(\frac{\widetilde{\L}}{|2\widetilde{T}|}+\frac{1-s}{2}\right)},
\end{equation*}
 where $\widetilde{T}=\frac{\partial}{\partial \tau}$. As in \cite{C82} the representation of the fundamental solution in this case is given by 
\begin{equation}\label{bs}
    \tilde{\mathcal{L}}_{s}^{-1}\delta_{0}=b_{s}|x|_{1}^{2s-Q},\,\,\,\,\,b_{s}=\frac{2^{n-1-s}\Gamma^{2}\left(\frac{Q-2s}{4}\right)}{\pi^{n+1}\Gamma(s)},
\end{equation}
where $\delta_{0}$ is the Dirac delta  at the point $0.$

To show the relation between $\Ls$ and $\tilde{\Ls}$, let us take $u(\xi,\tau)=v(2\xi,\tau)$. Then we have
\[\tilde{\mathcal{L}}u(\xi,\tau)=(\mathcal{L}v)(2\xi,\tau)\,,\]
and consequently also

$$\tilde{\mathcal{L}_{s}}u(\xi,\tau)=(\mathcal{L}_{s}v)(2\xi,\tau)\,,$$
where the last equality implies that

\begin{equation}\label{dil}
    \langle\tilde{\mathcal{L}_{s}}u,u\rangle_{L^{2}(\H)}=2^{-2n}\langle\mathcal{L}_{s}v,v\rangle_{L^{2}(\H)}.
\end{equation}

The auxiliary operator $U_s$, $s \in (0,1)$, was introduced in \cite{RT16} as being the following composition of fractional sub-Laplacians resulting in a bounded operator
\begin{equation}\label{Us}
    U_{s}=\Ls(\Lss)^{-1}\,.
\end{equation}
Its norm was computed there as 
\begin{equation}\label{Us.op.norm}
    \|U_{s}\|_{\rm op}=\sup_{k\geq0}\left(\frac{2k+n}{2}\right)^{-s}\frac{\Gamma\left(\frac{2k+n}{2}+\frac{1+s}{2}\right)}{\Gamma\left(\frac{2k+n}{2}+\frac{1-s}{2}\right)}.
\end{equation}
Giving an upper bound for $\|U_{s}\|_{\text{op}}$ that is convenient for our purposes and more explicit, and actually improves the one given in Subsection 5.3 in \cite{RT16}, requires the Wendel inequality for the ratio of Gamma functions, see \cite[Formula (7)]{W48}, which reads as follows
\begin{equation*}
    \frac{\Gamma(x+s)}{\Gamma(x)}\leq x^{s},\,\,\,\,\,\,x>0\,\,\,\,\text{and}\,\,\,\,\,s\in(0,1).
\end{equation*}
Choosing $x=\frac{2k+n+1-s}{2}$ and $s\in(0,1)$, we get
\begin{equation*}
    \begin{split}
        \left(\frac{2k+n}{2}\right)^{-s}\frac{\Gamma\left(\frac{2k+n}{2}+\frac{1+s}{2}\right)}{\Gamma\left(\frac{2k+n}{2}+\frac{1-s}{2}\right)}&=\left(\frac{2k+n}{2}\right)^{-s}\frac{\Gamma\left(\frac{2k+n+1-s}{2}+s\right)}{\Gamma\left(\frac{2k+n+1-s}{2}\right)}\\&
        \leq\left(\frac{2k+n}{2}\right)^{-s}\left(\frac{2k+n+1-s}{2}\right)^{s}\\&
        =\left(1+\frac{1-s}{2k+n}\right)^{s}\\&
        \leq \left(\frac{n+1-s}{n}\right)^{s}.
    \end{split}
\end{equation*}
A combination of the latter computations together with \eqref{Us.op.norm} yields
\begin{equation}\label{est1}
    \|U_{s}\|_{\text{op}}\leq \left(\frac{n+1-s}{n}\right)^{s}.
\end{equation}

In \cite{RT16} the authors also considered the auxiliary bounded operator $V_{s}=\mathcal{L}^{-1}_{1-s}\mathcal{L}\mathcal{L}^{-s}$, for $s \in (0,1)$. By the properties of the gamma function, one has the following estimate for the norm of $V_s$:
\begin{equation}\label{est2}
    \|V_{s}\|_{\text{op}}:=\|\mathcal{L}^{-1}_{1-s}\mathcal{L}\mathcal{L}^{-s}\|_{\text{op}}\leq \frac{n+2-s}{n+s}\,,
\end{equation}
 see \cite[Subsection 5.3]{RT16}.

We conclude this section with the next result which is a diagonalised version of the Hardy-Littlewood-Sobolev inequality with the best constant on $\H$ as shown in \cite{FL12}. 
\begin{thm}[The Hardy-Littlewood-Sobolev inequality]\label{H.liitle}
    Let $0 < \lambda < Q$ and $p := \frac{2Q}{2Q-\lambda}$. Then for
any $f,g \in L^p(\H)$ we have
\begin{equation*}
\begin{split}
    \int_{\H}\int_{\H}\frac{g(y)\overline{f(x)}}{|y^{-1}\circ_{1}x|_{1}^{\lambda}}dxdy&\leq C_{\text{HLS},\lambda}\|f\|_{L^{p}(\H)}\|g\|_{L^{p}(\H)},
\end{split}
\end{equation*}
where $|\cdot|_1$ is the Kaplan norm given in \eqref{fln}, and the best constant $C_{\textit{HLS},\lambda}$ is given by
\begin{equation*}
    C_{\text{HLS},\lambda}=\left(\frac{\pi^{n+1}}{2^{n-1}n!}\right)^{\frac{\lambda}{Q}}\frac{n!\Gamma\left(\frac{Q-\lambda}{2}\right)}{\Gamma^{2}\left(\frac{2Q-\lambda}{4}\right)}.
\end{equation*}
\end{thm}

\section{Fractional (logarithmic) Sobolev inequalities on $\mathbb{H}^{n}$}\label{sec:LS}

In this section, we show the fractional and the modified fractional (with respect to the operator $\Ls$ as in \eqref{Ls.Thang}) Sobolev inequalities on $\H$. We note that the first order Sobolev inequality on $\H$ with best constant was obtained by Jerison and Lee in \cite{JL88}, and here we extend their results for the modified fractional sub-Laplacian and for fractional powers of it. The log-Sobolev inequalities are also obtained, as well as the ``horizontal'' log-Sobolev inequality that eventually leads to the analogue of the Gross inequality on $\H$.

\begin{thm}[Fractional Sobolev inequality]\label{thm1}
   Let $s\in (0,1]$. The (modified) fractional Sobolev inequality on $\H$ is given by
    \begin{equation}\label{1b}
        \|f\|^{2}_{L^{\frac{2Q}{Q-2s}}}\leq C_{B,s}\langle \Ls f,f\rangle_{L^{2}(\H)},
    \end{equation}
    where 
    \begin{equation}\label{bestfrac}
        C_{B,s}=\frac{2^{-2s}\pi^{-s}(n!)^{\frac{s}{n+1}}\Gamma^{2}\left(\frac{Q-2s}{4}\right)}{\Gamma^{2}\left(\frac{Q+2s}{4}\right)}.
    \end{equation}
    Alternatively, inequality \eqref{1b} can be expressed in terms of fractional powers of $\L$ as 
    \begin{equation}\label{2b}
    \begin{split}
        \|f\|^{2}_{L^{\frac{2Q}{Q-2s}}}\leq C_{B,s}\|U_{s}\|_{\rm op} \langle \Lss f,f\rangle_{L^{2}(\H)},  
        \end{split}
    \end{equation}
  %  and
   % \begin{equation}\label{2b}
    %\begin{split}
     %    \|u\|^{2}_{L^{\frac{2Q}{Q-2s}}}&\leq \frac{\pi^{-s}(n!)^{\frac{s}{n+1}}2^{-\frac{2s}{n+1}}\Gamma^{2}\left(\frac{Q-2s}{4}\right)}{\Gamma^{2}\left(\frac{Q+2s}%{4}\right)}\|V_{s}\|_{\text{OP}}\langle\Lss u,u\rangle_{L^{2}(\H)},
    %\end{split}
    %\end{equation}
    %where $f\in H^{s}(\mathbb{H}^{n})$
    where $\|U_{s}\|_{\rm op}$ has been estimated in \eqref{est1}. 
\end{thm}

\begin{proof}
By using Parsevals's theorem and \cite[Corollary 4.1]{GK}, we have
\begin{equation*}
    \begin{split}
        |\langle u,g\rangle_{L^{2}(\H)}|^{2}&=\left|\int_{\mathbb{R}\setminus \{0\}}\text{Tr}\left(\pi_{\lambda}(u)\pi_{\lambda}(g)^{*}\right)d\mu(\lambda)\right|^{2}\\&
            =\left|\int_{\mathbb{R}\setminus \{0\}}\text{Tr}\left(\sigma_{\tilde{\Ls}}^{\frac{1}{2}}(\lambda)\pi_{\lambda}(u)\pi_{\lambda}(g)^{*}\sigma_{\tilde{\Ls}}^{-\frac{1}{2}}(\lambda)\right)d\mu(\lambda)\right|^{2}\\&
        \leq \left(\int_{\mathbb{R}\setminus \{0\}}\|\sigma_{\tilde{\Ls}}^{\frac{1}{2}}(\lambda)\pi_{\lambda}(u)\|^{2}_{\text{HS}}d\mu(\lambda)\right)\left(\int_{\mathbb{R}\setminus \{0\}}\|\sigma_{\tilde{\Ls}}^{-\frac{1}{2}}(\lambda)\pi_{\lambda}(g)\|^{2}_{\text{HS}}d\mu(\lambda)\right)\\&
        =\|\tilde{\Ls}^{\frac{1}{2}}u\|^{2}_{L^{2}(\H)}\|\tilde{\Ls}^{-\frac{1}{2}}g\|^{2}_{L^{2}(\H)}\\&
        =\langle \tilde{\Ls} u,u\rangle_{L^{2}(\H)}\langle  g,\tilde{\Ls}^{-1}g\rangle_{L^{2}(\H)},
    \end{split}
\end{equation*}
where $d\mu=\frac{2^{n-1}}{\pi^{n+1}}|\lambda|^{n}d\lambda$.
Using the Hardy-Littlewood-Sobolev inequality as in Theorem \ref{H.liitle} we obtain, with $p=\frac{2Q}{Q+2s}$,
\begin{equation*}
    \begin{split}
        \langle  g,\tilde{\Ls}^{-1}g\rangle_{L^{2}(\H)}&=b_{s}\int_{\H}\int_{\H}\frac{g(y)\overline{g(x)}}{|y^{-1}\circ_{1}x|_{1}^{Q-2s}}dxdy\\&
         \leq b_{s}C_{\text{HLS},Q-2s}\|g\|^{2}_{L^{p}(\H)}\\&
         =\frac{2^{n-1-s}\Gamma^{2}\left(\frac{Q-2s}{4}\right)}{\pi^{n+1}\Gamma(s)}\left(\frac{\pi^{n+1}}{2^{n-1}n!}\right)^{\frac{Q-2s}{Q}}\frac{n!\Gamma\left(s\right)}{\Gamma^{2}\left(\frac{Q+2s}{4}\right)}\|g\|^{2}_{L^{p}(\H)}\\&
         =\frac{2^{-\frac{2s}{n+1}}\pi^{-s}(n!)^{\frac{s}{n+1}}\Gamma^{2}\left(\frac{Q-2s}{4}\right)}{\Gamma^{2}\left(\frac{Q+2s}{4}\right)}\|g\|^{2}_{L^{p}(\H)},
    \end{split}
\end{equation*}
where $p=\frac{2Q}{Q+2s}$, while the constant $b_{s}$ and  the distance function $|\cdot|_{1}$ are defined in \eqref{bs} and \eqref{fln}, respectively.
Hence, by the above we have
\begin{equation*}\label{2sc}
    \begin{split}
       \|u\|^{2q}_{L^{q}(\H)}&= |\langle u,|u|^{q}u^{-1}\rangle_{L^{2}(\H)}|^{2}\\&
       \leq \langle \tilde{\Ls} u,u\rangle_{L^{2}(\H)}\langle  |u|^{q}u^{-1},\tilde{\Ls}^{-1}|u|^{q}u^{-1}\rangle_{L^{2}(\H)}\\&
       \leq \frac{2^{-\frac{2s}{n+1}}\pi^{-s}(n!)^{\frac{s}{n+1}}\Gamma^{2}\left(\frac{Q-2s}{4}\right)}{\Gamma^{2}\left(\frac{Q+2s}{4}\right)}\langle \tilde{\Ls} u,u\rangle_{L^{2}(\H)} \|u\|^{2q-2}_{L^{q}(\H)},
    \end{split}
\end{equation*}
for $q=\frac{p}{p-1}=\frac{2Q}{Q-2s}$. We note that $|u|^q u^{-1}$ is well defined in view of $||u|^q u^{-1}|\leq |u|^{q-1}$, and $q>1$. The latter can be rewritten as 

\begin{equation*}
    \|u\|^{2}_{L^{\frac{2Q}{Q-2s}(\H)}}\leq \frac{2^{-\frac{2s}{n+1}}\pi^{-s}(n!)^{\frac{s}{n+1}}\Gamma^{2}\left(\frac{Q-2s}{4}\right)}{\Gamma^{2}\left(\frac{Q+2s}{4}\right)}\langle \tilde{\Ls} u,u\rangle_{L^{2}(\H)}. 
\end{equation*}

Now by \eqref{dil} and substituting $u(\xi,\tau_{1})=f(2\xi,\tau_{1})$ we get, noting $\frac{2Q}{Q-2s}=\frac{2n+2}{n+1-s}$,
\begin{equation}\label{new}
    \begin{split}
       \left(\int_{\H}2^{-2n}|f(\xi_{1},\tau_{1})|^{\frac{2n+2}{n+1-s}}d\xi_{1} d\tau_{1}\right)^{\frac{n+1-s}{n+1}}&=\left(\int_{\H}|f(2\xi,\tau_{1})|^{\frac{2n+2}{n+1-s}}d\xi d\tau_{1}\right)^{\frac{n+1-s}{n+1}}\\&
       =\left(\int_{\H}|u(\xi,\tau_{1})|^{\frac{2n+2}{n+1-s}}d\xi d\tau_{1}\right)^{\frac{n+1-s}{n+1}}=\\&
        \leq\frac{2^{-\frac{2s}{n+1}}\pi^{-s}(n!)^{\frac{s}{n+1}}\Gamma^{2}\left(\frac{Q-2s}{4}\right)}{\Gamma^{2}\left(\frac{Q+2s}{4}\right)}\langle \tilde{\Ls} u,u\rangle_{L^{2}(\H)}\\&
        =\frac{2^{-\frac{2s}{n+1}-2n}\pi^{-s}(n!)^{\frac{s}{n+1}}\Gamma^{2}\left(\frac{Q-2s}{4}\right)}{\Gamma^{2}\left(\frac{Q+2s}{4}\right)}\langle \Ls f,f\rangle_{L^{2}(\H)}\,.
    \end{split}
\end{equation}
We note that $2^{-\frac{2s}{n+1}-2n}2^{2n\frac{n+1-s}{n+1}}=2^{-2s}$.
Hence \eqref{new} implies
\begin{eqnarray*}
    \left(\int_{\H}|f(\xi_{1},\tau_{1})|^{\frac{2n+2}{n+1-s}}d\xi_{1} d\tau_{1}\right)^{\frac{n+1-s}{n+1}} &\leq & 2^{-2s} \frac{\pi^{-s}(n!)^{\frac{s}{n+1}}\Gamma^{2}\left(\frac{Q-2s}{4}\right)}{\Gamma^{2}\left(\frac{Q+2s}{4}\right)} \langle \Ls f,f\rangle_{L^{2}(\H)}\,,
\end{eqnarray*}
 the last can be rewritten as 
\begin{equation*}
    \|f\|^{2}_{L^{\frac{2Q}{Q-2s}}}\leq C_{B,s}\langle \Ls f,f\rangle_{L^{2}(\H)}\,,
\end{equation*}
with $C_{B,s}$ as in \eqref{bestfrac}. Equivalently, using the fact that $U_s=\Ls(\Lss)^{-1}$, the latter can be expressed as in \eqref{2b}, and the proof is complete.
\end{proof}
\begin{rem}\label{remJL}
We note that when  $s=1$ in Theorem \ref{thm1}, the appearing constant in the inequality \eqref{1b} coincides with the one in the Sobolev inequality with respect to the sub-Laplacian involving the vector fields $\tilde{X}_{i}$ and $\tilde{Y}_{i},$ for $i=1,\ldots,n$, that was shown by Jerison and Lee in \cite{JL88}. Additionally, it is clear from the proof of Theorem \ref{thm1} and the sharpness of the Sobolev inequality involving the modified sub-Laplacian in \cite{JL88} that the Sobolev inequality involving the sub-Laplacian $\Lss$ is also sharp in general, and for the appearing constant $C_{B,s}$ for $s=1$, one has   
\begin{equation}\label{bestintsob}
    C_{B,1}=\frac{(n!)^{\frac{1}{n+1}}}{\pi n^{2}}\,,
\end{equation}
since 
\[
\frac{\Gamma\left(\frac{Q-2}{4} \right)}{\Gamma\left(\frac{Q+2}{4} \right)}=\frac{\Gamma\left(\frac{Q-2}{4} \right)}{\Gamma\left(\frac{Q-2}{4}+1 \right)}=\frac{4}{Q-2}=\frac{2}{n}\,,
\]
where we have used the property that for $x>0$ we have $\Gamma(x+1)=x \Gamma(x)$.
\end{rem}

\begin{rem}
    In \cite[p. 151]{RT16} the authors deduce a weaker form of Hardy inequality which is essentially the (modified) fractional inequality on $\H$ and conclude that the resulting inequality is not the one obtained by Frank and Lieb in \cite[p.353]{FL12}. However, comparing the latter inequality with the one in \cite{RT16} for $s=1$, one can observe that there is the extra constant $2^{-\frac{1}{n+1}}$ on the left-hand side of the inequality in \cite{RT16}. The authors of \cite{RT16} have confirmed to us that this factor arose from a confusion related to the use of two different sub-Laplacian in these two papers, and should not be there. The arguments in the proof of Theorem \ref{thm1} take this fact into account, and yield the correct constant in these inequalities.
    
    %According to the authors of the latter work, this is due to the missing factor $2^{\frac{s}{n+1}}$ on the right-hand side of the Hardy inequality as in \cite[p. 150]{RT16}.\textcolor{blue}{Michael, can you please rewrite this part?}
\end{rem}

To obtain the logarithmic analogues of the (fractional) Sobolev inequalities as in Theorem \ref{thm1}, the latter must be combined with the logarithmic H\"{o}lder's inequality on general measure spaces. The following lemma states the desired inequality that was first shown in \cite{CKR21c}; see also \cite[Lemma 3.2]{KRS20} for a more general version of it.
\begin{lem}[Logarithmic H\"older inequality]\label{holder}
Let $\X$ be a  measure space and let $u\in L^{p}(\mathbb{X})\cap L^{q}(\mathbb{X})\setminus\{0\}$, where $1<p<q< \infty.$ 
Then we have
\begin{equation}\label{holdernn}
\int_{\mathbb{X}}\frac{|u|^{p}}{\|u\|^{p}_{L^{p}(\mathbb{X})}}\log\left(\frac{|u|^{p}}{\|u\|^{p}_{L^{p}(\mathbb{X})}}\right)dx\leq \frac{q}{q-p}\log\left(\frac{\|u\|^{p}_{L^{q}(\mathbb{X})}}{\|u\|^{p}_{L^{p}(\mathbb{X})}}\right).
\end{equation}
\end{lem}
A combination of Lemma \ref{holder} together with Theorem \ref{thm1} yields the logarithmic version of the Sobolev inequality involving both the modified fractional sub-Laplacian $\Ls$ and powers of the sub-Laplacian $\Lss$. 
\begin{thm}[Logarithmic Sobolev inequality]\label{thm:Logsob}
Let $s \in (0,1]$. The (modified) fractional logarithmic Sobolev inequality on $\H$ is as follows   
   \begin{equation}\label{LogSobolev1}
\int_{\H}\frac{|u|^{2}}{\|u\|^{2}_{L^{2}(\H)}}\log\left(\frac{|u|^{2}}{\|u\|^{2}_{L^{2}(\H)}}\right) dx \leq \frac{Q}{2s}\log\left(C_{B,s}\frac{\langle \Ls u,u\rangle_{L^{2}(\H)}}{\|u\|^{2}_{L^{2}(\H)}}\right),
\end{equation}
where $C_{B,s}$ is given in \eqref{bestfrac}.
   Additionally, the logarithmic Sobolev inequality on $\H$ in terms of fractional powers of $\L$ reads as follows 
    \begin{equation}\label{LogSobolev2}
\int_{\H}\frac{|u|^{2}}{\|u\|^{2}_{L^{2}(\H)}}\log\left(\frac{|u|^{2}}{\|u\|^{2}_{L^{2}(\H)}}\right) dx \leq \frac{Q}{2s}\log\left(C_{B,s}\|U_{s}\|_{\rm op}\frac{\langle \Lss u,u\rangle_{L^{2}(\H)}}{\|u\|^{2}_{L^{2}(\H)}}\right),
\end{equation}
  where $\|U_{s}\|_{\rm op}$ has been estimated in \eqref{est1}.
\end{thm}
\begin{proof}
From the logarithmic H\"older inequality \eqref{holdernn} we have
\begin{equation}
\int_{\H}\frac{|u(x)|^{2}}{\|u\|^{2}_{L^{2}(\H)}}\log\left(\frac{|u(x)|^{2}}{\|u\|^{2}_{L^{2}(\H)}}\right)dx\leq \frac{q}{q-2}\log\left(\frac{\|u\|^{2}_{L^{q}(\H)}}{\|u\|^{2}_{L^{2}(\H)}}\right).
\end{equation}
For $p=2<q=p^{*}=\frac{2Q}{Q-2s}$ by \eqref{1b} we have
\begin{eqnarray*}\label{log.sob.}
\int_{\H}\frac{|u(x)|^{2}}{\|u\|^{2}_{L^{2}(\H)}}\log\left(\frac{|u(x)|^{2}}{\|u\|^{2}_{L^{2}(\H)}}\right) dx & \leq & \frac{q}{q-2}\log\left(\frac{\|u\|^{2}_{L^{q}(\H)}}{\|u\|^{2}_{L^{2}(\H)}}\right)\nonumber\\
& \leq  & \frac{q}{q-2} \log\left(C_{B,s}\frac{\langle \Ls u,u\rangle_{L^{2}(\H)}}{\|u\|^{2}_{L^{2}(\H)}}\right)\nonumber\\
& =  & \frac{Q}{2s} \log\left(C_{B,s}\frac{\langle \Ls u,u\rangle_{L^{2}(\H)}}{\|u\|^{2}_{L^{2}(\H)}}\right),
\end{eqnarray*}
where $C_{B,s}$ is defined in \eqref{bestfrac},  and this shows \eqref{LogSobolev1}. In the same way, we get \eqref{LogSobolev2}, and the proof of the theorem is complete.\end{proof}
As the next result shows one can rewrite the fractional logarithmic Sobolev inequality \eqref{LogSobolev1} for $s=1$ using the horizontal gradient on $\H$. The proof of Corollary \ref{cor.s=1} is immediate since $\L=-\nabla_{\H}^{*}\nabla_{\H}$. 
\begin{cor}\label{cor.s=1}
    We have
     \begin{equation}\label{LogSobolevint}
\int_{\H}\frac{|u|^{2}}{\|u\|^{2}_{L^{2}(\H)}}\log\left(\frac{|u|^{2}}{\|u\|^{2}_{L^{2}(\H)}}\right) dx \leq \frac{Q}{2}\log\left(C_{B,1}\frac{\|\nabla_{\H}u\|^{2}_{L^{2}(\H)}}{\|u\|^{2}_{L^{2}(\H)}}\right)\,.
\end{equation}
Additionally, whenever $u$ is such that $\|u\|_{L^{2}(\H)}=1$, we have
    \begin{equation}\label{u=1}
        \int_{\H}|u|^{2}\log\left(|u|\right) dx \leq \frac{Q}{4}\log\left(C_{B,1}\|\nabla_{\H}u\|^{2}_{L^{2}(\H)}\right),
    \end{equation}
where $C_{B,1}$ is given in \eqref{bestintsob}.
\end{cor}

%\begin{rem}
 %   As noted in Remark \ref{remJL} the constant $C_{B,1}$ is sharp. Consequently the log-Sobolev inequalities \eqref{LogSobolev1} and \eqref{u=1} involving the sub-gradient $\nabla_{\H}$ are sharp as well.\textcolor{blue}{Why not?}
%\end{rem}

\section{Gross inequalities}

In this section, we prove that the analogue of the Gross inequality  holds true in the $\H$-setting. Even more, in Theorem \ref{equiv.thm} we show that the log-Sobolev inequality involving the horizontal gradient, the so-called ``horizontal'' log-Sobolev inequality, \eqref{LogSobolev1} is equivalent to the latter, as it happens in the Euclidean setting, see \cite{Bec98}.  The idea of passing the analogue of the Gross inequality on $\H$ to an infinite dimensional object like $\mathbb{H}^{\infty}$ is discussed in Observation \ref{mainrem}
where we see that in the limiting case, i.e., when the dimension of the first stratum of $\H$ becomes very big, the probability measure that is considered on it, becomes exactly the Gaussian measure on $\mathbb{R}^{2n}$.

\begin{thm}[Semi-Gaussian log-Sobolev inequality on $\H$]\label{semi-g}
The following ``semi-Gaussian'' log-Sobolev inequality is satisfied
\begin{equation}
    \label{gaus.log.sob}
    \int_{\H}|g|^2\log|g|\,d\mu \leq \int_{\H} |\nabla_{\H}g|^2\,d\mu\,,
\end{equation}
for any $g$ such that $\|g\|_{L^2(\H,\mu)}=1,$
 where $\mu=\mu_1 \otimes \mu_2$, and $\mu_1$ is the Gaussian measure on $\mathbb{R}^{2n}$ given by $d\mu_1=\gamma e^{-\frac{|\xi|^2}{2}}d\xi$, for $\xi\in \mathbb{R}^{2n}$, with $|\xi|$ being the Euclidean norm of $\xi$, where 
 \begin{equation}
     \label{k}
     \gamma=n!\left(\frac{n+1}{2\pi n^{2}}\right)^{n+1}e^{n-1}\,,
 \end{equation}
 and  $d\mu_2=d\tau$ is the Lebesgue measure on $\mathbb{R}$ with respect to the variable $\tau$, for $x=(\xi,\tau) \in \H$.
\end{thm}
As mentioned already in the introduction the ``semi-Gaussian'' inequality comes along with the idea of extending it to the infinite-dimensional Heisenberg group $\mathbb{H}^{\infty}$. The following observation argues in this direction. Very interestingly it shows that in the limiting case when $n \rightarrow \infty$ the probability measure in the case of $\H$ behaves exactly like the Gaussian measure on $\mathbb{R}^{2n}$. This means that the analogue of the Gross inequality in the case of $\mathbb{H}^{\infty}$ is the exact extension of the Gross inequality on $\H$.
\begin{ob}\label{mainrem}
First, observe that for large $n\gg 1$ (that is $Q\gg1$), we have that the constant $\gamma$ has the following form
\[
\gamma_{n \gg 1}\simeq(2\pi)^{-n-\frac{1}{2}}n^{-\frac{1}{2}}=\frac{1}{(2\pi)^n}\frac{1}{\sqrt{2\pi n}}\,.
\]
Indeed by using Stirling's approximation formula (see e.g. \cite{Rom18}), i.e., that for large $n$ we have $n! \sim \sqrt{2 \pi n} \left( \frac{n}{e}\right)^{n}$, one can compute 
\begin{eqnarray*}
    \gamma & = & n!\left(\frac{n+1}{2\pi n^{2}}\right)^{n+1}e^{n-1}\\
    & = & n! e^{n-1}(2\pi n)^{-n-1}\left(1+\frac{1}{n} \right)^{n+1}\\
    & \simeq & \sqrt{2\pi n}\left( \frac{n}{e}\right)^n e^{n-1}(2\pi n)^{-n-1} e \\
    & = & (2\pi)^{-n-\frac{1}{2}}n^{-\frac{1}{2}}\,,
\end{eqnarray*}
where we have used the fact that $\lim\limits_{n \rightarrow \infty} \left(1+\frac{1}{n} \right)^{n+1}=e$.
 
    Now, since the normalisation constant $\gamma$ corresponds to the probability measure of the first stratum of topological dimension $2n$, each coordinate $\xi_i$, $i=1,\cdots, 2n$, in it  should correspond to a probability measure $d\mu_{1,j}$, $j=1,\cdots, 2n$, with normalisation constant $\gamma_{n \gg 1, j}=\left(\gamma_{n \gg 1}\right)^{\frac{1}{2n}}$, so that 
    \[
    d\mu_1=d\mu_{1,1} \times \cdots \times d\mu_{1,2n}:= \left(\gamma_{n \gg 1, 1} e^{-\frac{\xi_{1}^{2}}{2}}d\xi_1\right)\times \cdots \times \left(\gamma_{n \gg 1, 2n} e^{-\frac{\xi_{2n}^{2}}{2}}d\xi_{2n}\right)\,.
    \]
    Hence, in order to be able to extend \eqref{gaus.log.sob} to $\mathbb{H}^{\infty}$ with infinite dimensional first stratum we must have $\lim\limits_{n \rightarrow \infty} \left(\gamma_{n \gg 1}\right)^{\frac{1}{2n}}=c$, for some $c>0$. Indeed we have 
    \begin{eqnarray}\label{ninfinity}
      \lim_{n \rightarrow \infty} \left(\gamma_{n \gg 1}\right)^{\frac{1}{2n}} & = & \lim_{n \rightarrow \infty} \left((2\pi)^{-n-\frac{1}{2}}n^{-\frac{1}{2}}\right)^{\frac{1}{2n}} \nonumber \\
      & = & (2\pi)^{-\frac{1}{2}}  \lim_{n \rightarrow \infty} (2\pi)^{-\frac{1}{4n}}n^{-\frac{1}{4n}} \nonumber\\
      & = & (2\pi)^{-\frac{1}{2}}  \lim_{n \rightarrow \infty}n^{-\frac{1}{4n}}= (2\pi)^{-\frac{1}{2}}\,,
    \end{eqnarray}
    since $\lim\limits_{n \rightarrow \infty}(2\pi)^{-\frac{1}{4n}}=1$ and \[\lim_{n \rightarrow \infty}n^{-\frac{1}{4n}}=\lim_{n \rightarrow \infty}e^{\log n^{-\frac{1}{4n}}}= \lim_{n \rightarrow \infty}e^{-\frac{1}{4n}\log n}=1\,. \]
Equality \eqref{ninfinity} shows not only that $\lim\limits_{n \rightarrow \infty} \left(\gamma_{n \gg 1}\right)^{\frac{1}{2n}}=\frac{1}{\sqrt{2\pi}}$, as required, but also it shows that the probability measure on the first stratum of $\H$ behaves in the limiting case when $n \rightarrow \infty$ exactly like $\mathbb{R}^{2n}$. Recall, see \cite{Gro92}, that the normalisation constant in the classical Gross inequality is $(2\pi)^{-\frac{1}{2n}}$ on the whole $\mathbb{R}^n$, meaning that the probability measure corresponding to each of the $n$ coordinates is $(2\pi)^{-\frac{1}{2}}$ with the measure $\prod\limits_{j=1}^{n}\left(\frac{1}{\sqrt{2\pi}}e^{x_{j}^{2}}dx_{j}\right)$; exactly as it happens in the case of $\H$ when $n \rightarrow \infty$.
\end{ob}

\begin{proof}[Proof of Theorem \ref{semi-g}]
Assume that $g\in C^{\infty}_{0}(\H)$ is a compactly supported function that satisfies $\|g\|_{L^2(\H,\mu)}=1,$ where $\mu$ is the measure given in the hypothesis. Let us define $f(x)$ by $$f(x)=\gamma^{\frac{1}{2}}e^{-\frac{|\xi|^2}{4}}g(x),\,\,\,\,x=(\xi,\tau)\in\H,$$ 
where $\gamma$ is as in \eqref{k} and $|\xi|^{2}=\sum\limits_{k=1}\limits^{2n}\xi^{2}_{k}$ for $\xi\in \mathbb{R}^{2n}$.  Then clearly $f\in C^{\infty}_{0}(\H)$, while it is easy to check that $\|f\|_{L^2(\H)}=1$.
Indeed we have 
 \begin{equation}\label{norms1}
1=\|g\|_{L^2(\H,\mu)}^2=\int_{\H}\gamma^{-1} e^{\frac{|\xi|^2}{2}}|f(x)|^2\,d\mu=\int_{\H}|f(x)|^2\,dx\,.
 \end{equation}
An application of the log-Sobolev inequality \eqref{LogSobolevint} yields 
\begin{eqnarray}\label{thm.eq.glogg}
\int_{\H}|g(x)|^2\log |g(x)|\,d\mu & = & \int_{\H} |f(x)|^2 \log |\gamma^{-\frac{1}{2}}e^{\frac{|\xi|^2}{4}}f(x)|\,dx \nonumber\\
& \leq & \frac{Q}{4}\log \left(C_{B,1} \int_{\H}|\nabla_{\H}f(x)|^2\,dx\right) \nonumber \\
& & +\log (\gamma^{-\frac{1}{2}})+\int_{\H}\frac{|\xi|^2}{4} |f(x)|^2\,dx\,.
\end{eqnarray}
For $\xi\in \mathbb{R}^{2n}$, it is worth recalling that
\[
X_{i}=\frac{\partial}{\partial\xi_{i}}+\frac{1}{2}\xi_{n+i}\frac{\partial}{\partial \tau}, \,\,\, Y_{i}=\frac{\partial}{\partial \xi_{n+i}}-\frac{1}{2}\xi_{i}\frac{\partial}{\partial \tau}\,, \quad  i=1,\ldots,n\,. 
\]
For each $i=1,\ldots,n$, we have
\begin{eqnarray}\label{Xg}
|X_ig(x)|^2 &=& \gamma^{-1} e^{\frac{|\xi|^2}{2}} \left|X_if(x)+\frac{\xi_{i}}{2}f(x) \right|^{2}\nonumber\\
&=& \gamma^{-1} e^{\frac{|\xi|^2}{2}} \left(|X_if(x)|^2+\frac{(\xi_{i})^2}{4}|f(x)|^2+{\rm Re} \overline{(X_if(x))}\xi_{i}f(x) \right)\,,
\end{eqnarray}
and 
\begin{eqnarray*}\label{Yg}
|Y_ig(x)|^2 &=& \gamma^{-1} e^{\frac{|\xi|^2}{2}} \left|Y_if(x)+\frac{\xi_{n+i}}{2}f(x) \right|^{2}\nonumber\\
&=& \gamma^{-1} e^{\frac{|\xi|^2}{2}} \left(|Y_if(x)|^2+\frac{(\xi_{n+i})^2}{4}|f(x)|^2+{\rm Re} \overline{(Y_if(x))}\xi_{n+i}f(x) \right)\,.
\end{eqnarray*}
Moreover, notice that for each $\xi_{i}$,  $i=1,\ldots,2n$, we have
\begin{eqnarray*}
{\rm Re}\int_{\H}\overline{(\partial_{\xi_{i}}f(x))}\xi_{i}f(x)\,dx & =& -{\rm Re}\int_{\H}(\partial_{\xi_{i}}f(x))\xi_{i}\overline{f(x)}\,dx-\int_{\H}|f(x)|^2\,dx\\
&=& -{\rm Re}\int_{\H}\overline{(\partial_{\xi_{i}}f(x))}\xi_{i}f(x)\,dx-1\,.
\end{eqnarray*}
The above computations imply that for each $i=1,\ldots,2n$ we have
\begin{equation}
    \label{int.parts1}
    {\rm Re}\int_{\H}\overline{(\partial_{\xi_{i}}f(x))}\xi_{i}f(x)\,d\xi d\tau=-\frac{1}{2}\,.
\end{equation}
To explicitly compute $\int_{\H}{\rm Re}\overline{(X_if(x))}\xi_{i}f(x)\,dx$ it remains to compute the term \[\frac{1}{2}{\rm Re} \int_{\H}\xi_{n+i}\overline{(\partial_{\tau}f(x))}\xi_{i}f(x)dx\,\,\,,i=1,\ldots,n.\] To this end we have
\begin{eqnarray*}
   \frac{1}{2}{\rm Re} \int_{\H}\xi_{n+i}\overline{(\partial_{\tau}f(x))}\xi_{i}f(x)dx& = & - \frac{1}{2}{\rm Re}\int_{\H}\partial_{\tau}((\xi_{n+i}f(x)\xi_{i})\overline{f(x)}\,dx\\
    &=&- \frac{1}{2}{\rm Re}\int_{\H}\xi_{n+i}\overline{(\partial_{\tau}f(x))}\xi_{i}f(x)\,dx\,,
\end{eqnarray*}
implying that 
\begin{equation}  \label{int.parts2}
   {\rm Re} \int_{\H}\xi_{n+i}\overline{(\partial_{\tau}f(x))}\xi_{i}f(x)\,dx=0,\,\,\,i=1,\ldots,n.
\end{equation}
A similar argument applies to $Y_i$.
Summarising, by \eqref{int.parts1} and \eqref{int.parts2} we get 
\[
\int_{\H}{\rm Re}\overline{(X_if(x))}\xi_{i}f(x)\,dx=-\frac{1}{2}\,,
\]
and similarly
\[
\int_{\H}{\rm Re}\overline{(Y_if(x))}\xi_{n+i}f(x)\,dx=-\frac{1}{2}\,,
\]
for each $i=1,\ldots,n$. By the above and expression \eqref{Xg} we get 
\begin{equation}\label{thm.eq.nablag}
\int_{\H}|\nabla_{\H}g(x)|^2\,d\mu=\int_{\H}|\nabla_{\H}f(x)|^2\,dx+\int_{\H}\frac{|\xi|^2}{4}|f(x)|^2\,dx-n\,.
\end{equation}
Now, taking into account \eqref{thm.eq.glogg} and \eqref{thm.eq.nablag}, we see that showing \eqref{gaus.log.sob} reduces to proving that 
\[
\frac{Q}{4}\log \left(C_{B,1} \int_{\H}|\nabla_{\H}f(x)|^2\,dx\right)+\log (\gamma^{-1/2}) \leq \int_{\H}|\nabla_{\H}f(x)|^2\,dx-n\,,
\]
or
\[
\frac{Q}{4}\log \left( C_{B,1} \int_{\H}|\nabla_{\H}f(x)|^2\,dx \right)+\log (\gamma^{-\frac{1}{2}}) + \log e^{n}\leq \int_{\H}|\nabla_{\H}f(x)|^2\,dx\,.
\]
The latter can be rewritten as 
\[
\frac{Q}{4} \log \left(C_{B,1} \gamma^{-\frac{2}{Q}} e^{\frac{4n}{Q}} \int_{\H}|\nabla_{\H}f(x)|^2\,dx  \right)\leq \int_{\H}|\nabla_{\H}f(x)|^2\,dx\,,
\]
or  as
\[
\log \left(C_{B,1} \gamma^{-\frac{2}{Q}} e^{\frac{4n}{Q}} \int_{\H}|\nabla_{\H}f(x)|^2\,dx  \right)\leq \frac{4}{Q} \int_{\H}|\nabla_{\H}f(x)|^2\,dx\,.  
\]
By the above and the property that for all $r>0$ we have $\log r \leq r-1$, we conclude that showing \eqref{gaus.log.sob} amounts to proving the following
\begin{equation}\label{s-gaus.f}
{e^{-1}}C_{B,1}\gamma^{-\frac{2}{Q}}e^{\frac{4n}{Q}} \int_{\H}|\nabla_{\H}f(x)|^2\,dx \leq \frac{4}{Q} \int_{\H}|\nabla_{\H}f(x)|^2\,dx\,.
\end{equation}
Indeed, observe that \eqref{s-gaus.f} holds true even as an equality for the choice of $\gamma$ as in \eqref{k}. Hence, we have shown the desired inequality, i.e., that 
\begin{equation*}
\begin{split}
&\log \left(C_{B,1} \gamma^{-\frac{2}{Q}} e^{\frac{4n}{Q}} \int_{\H}|\nabla_{\H}f(x)|^2\,dx  \right)\\&  =  \log \left(e^{-1} \frac{(n!)^{\frac{1}{n+1}}}{\pi n^{2}}\gamma^{-\frac{1}{n+1}}e^{\frac{2n}{n+1}} \int_{\H}|\nabla_{\H}f(x)|^2\,dx\right)+1\\&
 \leq  e^{-1} \frac{(n!)^{\frac{1}{n+1}}}{\pi n^{2}}\gamma^{-\frac{1}{n+1}}e^{\frac{2n}{n+1}} \int_{\H}|\nabla_{\H}f(x)|^2\,dx\\&
=  \frac{4}{Q}  \int_{\H}|\nabla_{\H}f(x)|^2\,dx,
\end{split}
\end{equation*}
and so we have shown \eqref{gaus.log.sob} for $g \in C_{0}^{\infty}(\H)$. The proof is now complete by the density of $C_{0}^{\infty}(\H)$ in $L^2(\H, \mu)$.
\end{proof}
\begin{rem}
  As mentioned earlier, the logarithmic Sobolev inequality with respect to a probability measure is a key inequality in the field of infinite dimensional analysis, see e.g. \cite{BZ05} and references therein. The associated terminology suggests the following expression for the Gaussian log-Sobolev inequality \eqref{gaus.log.sob} 
    \begin{equation}\label{ent.gross}
    \frac{1}{2}\text{Ent}(|g|^2)\leq \int_{\H} |\nabla_{\H}g|^2\,d\mu\,, \quad \text{for}\quad \|g\|_{L^2(\H,d\mu)}=1\,,
    \end{equation}
    where the entropy functional $\text{Ent}(g)$ for a measurable function $g$ is defined as 
    \[
    \text{Ent}(g):=\int g \log g \,d\mu-\int g \,d\mu \log \int g\,d\mu\,.
    \]
    It is then clear that for $g$ such that $\|g\|_{L^2(\H,d\mu)}=1$, we have 
    \[
    \text{Ent}(|g|^2)=2 \int_{\H}|g|^2\log|g|\,d\mu\,,
    \]
    implying that the inequalities \eqref{ent.gross} and \eqref{gaus.log.sob} are euivalent.
\end{rem}

\begin{thm}\label{equiv.thm}
Let  $d\mu$ be the semi-Gaussian measure given in the hypothesis of Theorem \ref{semi-g}. The following statements are both true and imply each other:
\begin{enumerate}[label=\roman*.]
    \item[(i)] \label{11i} for $g$ such that $\|g\|_{L^2(\H,\mu)}=1,$ we have  \begin{equation*}
    \label{thm.eq.itmq}
    \int_{\H}|g|^2 \log |g|\,d\mu \leq \int_{\H} |\nabla_{\H}g|^2\,d\mu\,;\end{equation*}
    \item[(ii)] \label{2i} for $f$ such that $\|f\|_{L^2(\H)}=1,$ we have \begin{equation*}\label{thm.eq.itm2}\int_{\H}|f|^{2}\log |f|\,dx \leq \frac{Q}{4}\log \left(C_{B,1} \int_{\H} |\nabla_{\H}f|^2\,dx \right)\,,\end{equation*} where $C_{B,1}$ is as in \eqref{bestintsob}.
\end{enumerate}

\end{thm}
\begin{proof} The implication (ii) $\Rightarrow$ (i) is exactly the proof of Theorem \ref{semi-g}. Hence we will show that (i) $\Rightarrow$ (ii).
For $\gamma$ as in \eqref{k} and for $f$ as in (i), we define $g$ by 
\[
g(x)=\gamma^{-\frac{1}{2}}e^{\frac{|\xi|^2}{4}}f(x)\,.
\]
It is then clear that $\|g\|_{L^2(\H,\mu)}=1$, see \eqref{norms1}. We also have
\[
\int_{\H}|g(x)|^2\log |g(x)|\,d\mu  =  \int_{\H} |f(x)|^2 \log |\gamma^{-\frac{1}{2}}e^{\frac{|\xi|^2}{4}}f(x)|\,dx\,.
\]
Hence by the assumptions in (ii) and equality \eqref{thm.eq.nablag} we get 
\begin{equation}\label{before.eps}
\log (\gamma^{-\frac{1}{2}}e^{n})+\int_{\H}|f(x)|^2\log |f(x)|\,dx \leq \int_{\H} |\nabla_{\H}f(x)|^2\,dx\,,
\end{equation}
for any function $f$ such that $\|f\|_{L^2(\H)}=1$. Now, if for $\epsilon>0$ the mapping $\delta_\epsilon: \H \rightarrow \H$ denotes the dilations on $\H$, i.e., for $x=(\xi, \tau) \in \H$ we have $\delta_\epsilon(\xi,\tau)=(\epsilon \xi, \epsilon^2 \tau)$, then for $f_\epsilon(x):=\epsilon^{\frac{Q}{2}}f(\delta_{\epsilon}(x))$ we still have $\|f_{\epsilon}\|_{L^2(\H)}=1$. 
Now, since 
\[
\int_{\H}|f_\epsilon|^2 \log|f_\epsilon|\,dx=\int_{\H}|f|^2 \log\left|\epsilon^{\frac{Q}{2}}f \right|\,dx\,,
\]
and 
\[
\int_{\H}|\nabla_{\H}f_\epsilon|^2\,dx=\epsilon^2 \int_{\H}|\nabla_{\H}f|^2\,dx\,,
\]
an application of \eqref{before.eps} to the function $f_\epsilon$ yields
\begin{equation}\label{after.eps}
    \begin{split}
\int_{\H}&|f|^2\log |f|\,dx  \leq  \epsilon^2 \int_{\H} |\nabla_{\H}f|^2\,dx-\frac{Q}{2}\log \epsilon-\log (\gamma^{-\frac{1}{2}}e^{n})\\&
=\epsilon^2 \int_{\H} |\nabla_{\H}f|^2\,dx-\frac{Q}{2}\log \epsilon-\log \left(\left[n!e^{n-1}\left(\frac{(n+1)}{2\pi n^{2}}\right)^{n+1}\right]^{-\frac{1}{2}}e^{n} \right)\\&
 =\epsilon^2 \int_{\H} |\nabla_{\H}f|^2\,dx-\frac{Q}{2}\log \epsilon+\frac{n+1}{2} \log \left(\frac{(n+1)(n!)^{\frac{1}{n+1}}}{2\pi n^{2}e} \right)        
    \end{split}
\end{equation}
    for all $\epsilon>0$, where in the last inequality we have plugged in the expression for $\gamma$. Finally, minimizing the  right-hand side of \eqref{after.eps} over $\epsilon$ gives the desired inequality. Indeed for $\epsilon=\sqrt{\frac{Q}{4 \int_{\H}|\nabla_{\H}f(x)|^2\,dx}}$ we have
    \begin{eqnarray*}
    \int_{\H}|f|^2\log |f|\,dx & \leq & 
    \frac{Q}{4}-\frac{Q}{2}\log \left( \sqrt{\frac{Q}{4 \int_{\H}|\nabla_{\H}f(x)|^2\,dx}}\right)\\
    &+&\frac{n+1}{2} \log \left(\frac{(n+1)(n!)^{\frac{1}{n+1}}}{2\pi n^{2}e} \right)\\
& = &     \frac{n+1}{2}+\frac{n+1}{2} \log \left(4 \int_{\H}|\nabla_{\H}f|^2\,dx \right)- \frac{n+1}{2}\log 2(n+1)\\
    &+&\frac{n+1}{2} \log \left(\frac{(n+1)(n!)^{\frac{1}{n+1}}}{2\pi n^{2}e} \right)\\
    & = & \frac{n+1}{2} \log \left(\frac{(n!)^{\frac{1}{n+1}}}{\pi n^{2}}  \int_{\H}|\nabla_{\H}f|^2\,dx \right)\,\\
    &=&\frac{Q}{4}\log \left(C_{B,1}  \int_{\H}|\nabla_{\H}f|^2\,dx \right)\,,
    \end{eqnarray*}
    and this shows that (i) implies (ii). The proof of Theorem \ref{equiv.thm} is now complete.
\end{proof}

\section{(logarithmic) Gagliardo-Nirenberg inequalities on $\H$}

In this section, we prove the Gagliardo-Nirenberg inequality on $\H$, as well the ``the derivates-type Gagliardo-Nirenberg inequality'' on $\H$. As earlier in Section \ref{sec:LS}, we prove the aforementioned inequalities also with respect to the modified fractional sub-Laplacian $\Ls$ on $\H$.
Let us first recall the next auxiliary result as in \cite{CKR21c}.
\begin{lem}[An interpolation inequality for general measure spaces]
Let $\mathbb{X}$ be any measure space, and let $1<\theta_1\leq q\leq \theta_2\leq\infty$. Then for any  $u\in L^{\theta_1}(\mathbb{X})\cap L^{\theta_2}(\mathbb{X})$ with
\begin{equation*}
\frac{1}{q}=\frac{a}{\theta_1}+\frac{1-a}{\theta_2}\,,
\end{equation*}
for some $a \in [0,1]$, we have 
\begin{equation}\label{holder.gen}
\|u\|_{L^{q}(\mathbb{X})}\leq \|u\|^{a}_{L^{\theta_1}(\mathbb{X})}\|u\|^{1-a}_{L^{\theta_2}(\mathbb{X})}.
\end{equation}
    
\end{lem}
\begin{thm}[Gagliardo-Nirenberg inequality]\label{GN} Let $s\in(0,1]$ and $Q>2s$ satisfy the relation $\frac{1}{q}=a\left(\frac{1}{2}-\frac{s}{Q}\right)+\frac{1-a}{\sigma}$,  for some $\sigma\geq q\geq \frac{2Q}{Q-2s}$ and $a \in (0,1]$.
Then the Gagliardo-Nirenberg inequality on $\H$ with respect to the (modified) sub-Laplacian is given by
\begin{equation}\label{gnin1}
    \int_{\H}|u(x)|^{q}dx \leq C^{\frac{aq}{2}}_{B,s}\langle \Ls u,u\rangle_{L^{2}(\H)}^{\frac{aq}{2}}\|u\|^{(1-a)q}_{L^{\sigma}(\H)}\,,
\end{equation}
where $C_{B,s}$ is as in \eqref{bestfrac}.
Moreover, we can rewrite the inequality \eqref{gnin1} with respect to fractional powers of the sub-Laplacian $\Ls$ as follows
\begin{equation}\label{gnin2}
    \int_{\H}|u(x)|^{q}dx \leq C^{\frac{aq}{2}}_{B,s}\|U_{s}\|^{\frac{aq}{2}}_{\rm op}\langle \Lss u,u\rangle_{L^{2}(\H)}^{\frac{aq}{2}}\|u\|^{(1-a)q}_{L^{\sigma}(\H)},
\end{equation}
where  the operator norm $\|U_{s}\|_{\rm op}$ has been estimated in \eqref{est1}.
\end{thm}

\begin{proof}
    By using the H\"{o}lder inequality \eqref{holder.gen}, and the fractional Sobolev inequality as in \eqref{1b}, we have
    \begin{equation*}
        \begin{split}
             \int_{\H}|u(x)|^{q}dx&=   \int_{\H}|u(x)|^{aq}|u(x)|^{(1-a)q}dx\\&
             \leq\|u\|^{aq}_{L^{2^{*}}(\H)}\|u\|^{(1-a)q}_{L^{\sigma}(\H)} \\&
             \leq C^{\frac{aq}{2}}_{B,s}\langle \Ls u,u\rangle_{L^{2}(\H)}^{\frac{aq}{2}}\|u\|^{(1-a)q}_{L^{\sigma}(\H)}\,,
        \end{split}
    \end{equation*}
    and we have proved inequality \eqref{gnin1}. Inequality \eqref{gnin2} follows now by \eqref{gnin1} and the definition of the operator $U_s$ as in \eqref{Us}.
\end{proof}
One can repeat the main arguments and show the following generalisation of \eqref{gnin1}.  
\begin{thm}[Two derivative-type Gagliardo-Nirenberg inequality] 
 Let $1\geq s_{1}\geq s_{2}\geq0$ be such that $\frac{2Q}{Q-2s_{2}}\leq q\leq \frac{2Q}{Q-2s_{1}}$. Then, the Gagliardo-Nirenberg inequality with two derivatives on $\H$ is as follows 
\begin{equation}\label{GN4}
    \int_{\H}|u(x)|^{q}dx\leq C_{GN,s_1,s_2}\langle\mathcal{L}_{s_{1}}u,u\rangle_{L^{2}(\H)}^{\frac{Q(q-2)-2s_{2}q}{4(s_{1}-s_{2})}}\langle\mathcal{L}_{s_{2}}u,u\rangle_{L^{2}(\H)}^{\frac{s_{1}(q-2)-Q(q-2)}{4(s_{1}-s_{2})}},
\end{equation}
    where $C_{GN,s_1,s_2}=C_{B,s_1}^{\frac{qa}{2}} C_{B,s_2}^{\frac{q(1-a)}{2}}$ for $a=\frac{Q(q-2)-2s_{2}q}{2(s_{1}-s_{2})q}$, where $C_{B,s_i}$, $i=1,2$, is given in \eqref{bestfrac}. Additionally, the modified version of it is given by
    \begin{equation}\label{GN5}
    \begin{split}
         \int_{\H}|u(x)|^{q}dx&\leq C_{GN,s_1,s_2}\|U_{s_1}\|^{\frac{Q(q-2)-2s_{2}q}{4(s_{1}-s_{2})}}_{\rm op}\|U_{s_2}\|^{\frac{s_{1}(q-2)-Q(q-2)}{4(s_{1}-s_{2})}}_{\rm op}\\&
         \times\langle\mathcal{L}^{s_{1}}u,u\rangle_{L^{2}(\H)}^{\frac{Q(q-2)-2s_{2}q}{4(s_{1}-s_{2})}}\langle\mathcal{L}^{s_{2}}u,u\rangle_{L^{2}(\H)}^{\frac{s_{1}(q-2)-Q(q-2)}{4(s_{1}-s_{2})}}\,,
    \end{split}
\end{equation}
where the norms $\|U_{s_i}\|_{\rm op}$ for $i=1,2$ are given in \eqref{est1}.
\end{thm}
\begin{proof}
First, let us show that $a=\frac{Q(q-2)-2s_{2}q}{2(s_{1}-s_{2})q}\in [0,1]$. To this end, we define the function \[A(q)=\frac{Q(q-2)-2s_{2}q}{2(s_{1}-s_{2})q}=\frac{Q-2s_2}{2(s_1-s_2)}-\frac{Q}{(s_1-s_2)}\frac{1}{q}\,,\] for $\frac{2Q}{Q-2s_{2}}\leq q\leq \frac{2Q}{Q-2s_{1}}$ as in the hypothesis. It is then easy to see that since $A'(q)=\frac{Q}{s_1-s_2}\frac{1}{q^2} >0$, we have that 
\[
0=A\left( \frac{2Q}{Q-2s_{2}}\right) \leq A(q) \leq A\left( \frac{2Q}{Q-2s_{1}}\right)=1\,.
\]
and this proves our claim. 
 %   Firstly, let us show $a=\frac{Q(q-2)-2s_{2}q}{2(s_{1}-s_{2})q}\in (0,1]$. By $\frac{2Q}{Q-2s_{2}}\leq q\leq \frac{2Q}{Q-2s_{1}}$, we have
  %  \begin{equation*}
   %     \begin{split}
    %      a=\frac{Q(q-2)-2s_{2}q}{2(s_{1}-s_{2})q}&\geq\frac{1}{2(s_{1}-s_{2})q} \left(Q\left(\frac{2Q}{Q-2s_{2}}-2\right)-2s_{2}\frac{2Q}{Q-2s_{1}}\right)\\& 
     %   =\frac{1}{2(s_{1}-s_{2})q}\left(\frac{4s_{2}Q}{Q-2s_{2}}-\frac{4s_{2}Q}{Q-2s_{1}}\right)\\&
     %   =\frac{4s_{2}Q}{2(s_{1}-s_{2})q}\frac{2s_{1}-2s_{2}}{(Q-2s_{2})(Q-2s_{1})}\\&
     %   =\frac{4s_{2}Q}{q(Q-2s_{2})(Q-2s_{1})}\\&
     %   >0,
     %   \end{split}
   %\end{equation*}
   %and by $q\leq \frac{2Q}{Q-2s_{1}}$, we get
   %\begin{equation*}
   %    \begin{split}
   % a=\frac{Q(q-2)-2s_{2}q}{2(s_{1}-s_{2})q}&=\frac{1}{2(s_{1}-s_{2})}\left(Q\left(1-\frac{2}{q}\right)-2s_{2}\right)\\&
   % \leq \frac{1}{2(s_{1}-s_{2})}\left(Q\left(1-2\left(\frac{1}{2}-\frac{s_{1}}{Q}\right)\right)-2s_{2}\right)\\&
   % =1.
   %    \end{split}
   %\end{equation*}
    By the H\"older inequality \eqref{holder.gen} with $\theta_{1}=\frac{2Q}{Q-2s_{1}}$, $\theta_{2}=\frac{2Q}{Q-2s_{2}}$ and $a=\frac{Q(q-2)-2s_{2}q}{2(s_{1}-s_{2})q}\in [0,1]$, and Sobolev inequality \eqref{1b} we get
    \begin{equation}
        \begin{split}
            \int_{\H}|u(x)|^{q}dx&=\int_{\H}|u(x)|^{aq}|u(x)|^{(1-a)q}dx\\&
        \leq \left(\int_{\H}|u(x)|^{\theta_{1}}dx\right)^{\frac{qa}{\theta_{1}}}\left(\int_{\H}|u(x)|^{\theta_{2}}dx\right)^{\frac{q(1-a)}{\theta_{2}}}\\&
        \leq C_{B,s_1}^{\frac{qa}{2}}\langle\mathcal{L}_{s_{1}}u,u\rangle_{L^{2}(\H)}^{\frac{qa}{2}}C_{B,s_2}^{\frac{q(1-a)}{2}}\langle\mathcal{L}_{s_{2}}u,u\rangle_{L^{2}(\H)}^{\frac{q(1-a)}{2}}\\&
        =C_{GN,s_1,s_2}\langle\mathcal{L}_{s_{1}}u,u\rangle_{L^{2}(\H)}^{\frac{Q(q-2)-2s_{2}q}{4(s_{1}-s_{2})}}\langle\mathcal{L}_{s_{2}}u,u\rangle_{L^{2}(\H)}^{\frac{s_{1}(q-2)-Q(q-2)}{4(s_{1}-s_{2})}}\,,
        \end{split}
    \end{equation}
    where we have defined $C_{GN,s_1,s_2}:=C_{B,s_1}^{\frac{qa}{2}} C_{B,s_2}^{\frac{q(1-a)}{2}}$ and we have shown \eqref{GN4}. Inequality \eqref{GN5} follows.
\end{proof}

\begin{rem}[Sharp version of Gagliardo-Nirenberg inequalities on $\H$]\label{rem.GN}
The appearing constants in the Gagliardo-Nirenberg inequalities on $\H$ may be in general not sharp. However, in view of Remark  6.2 in \cite{RTY20} that relates the best constant in the Sobolev inequality with the one in the Gagliardo-Nirenberg inequality, the Gagliardo-Nirenberg inequality \eqref{GN4} when $s_1=1$, $s_2=0$ and $2\leq q \leq \frac{2Q}{Q-2}$ is sharp and the best constant $C_{GN,1,0}$ is given by
\[
C_{GN,1,0}=\left[C_{B,1}\frac{2q}{2q-Q(q-2)}\left( \frac{Q(q-2)}{2q-Q(q-2)}\right) \right]^{\frac{q}{2}}\,,
\]
where $C_{B,1}$ is the best constant in the Sobolev inequality on $\H$ given by \eqref{bestintsob}.
\end{rem}
The logarithmic version of the Galirado-Nirenberg inequalities \eqref{gnin1} and \eqref{gnin2} are given in the next theorem. Following similar arguments, one can get the logarithmic versions of the inequalities \eqref{GN4} and \eqref{GN5}. The latter is a simple exercise and its details are left for the interested reader.  
\begin{thm}[Logarithmic Gagliardo-Nirenberg inequality]\label{thmlogGN}
 Let $s\in(0,1]$ and $Q>2s$ satisfy the relation $\frac{1}{q}=a\left(\frac{1}{2}-\frac{s}{Q}\right)+\frac{1-a}{\sigma}$,  for some $\sigma\geq q\geq \frac{2Q}{Q-2s}$ and $a \in (0,1]$. Then the logarithmic version of the Gagliardo-Nirenberg inequality with respect to the (modified) fractional sub-Laplacian is as follows
   \begin{equation}\label{LogGN1}
\int_{\H}\frac{|u|^{2}}{\|u\|^{2}_{L^{2}(\H)}}\log\left(\frac{|u|^{2}}{\|u\|^{2}_{L^{2}(\H)}}\right) dx \leq \frac{q}{q-2}\log\left(C_{B,s}^{a}\frac{\langle \Ls u,u\rangle^{a}_{L^{2}(\H)}\|u\|^{2(1-a)}_{L^{\sigma}(\H)}}{\|u\|^{2}_{L^{2}(\H)}}\right)\,.
\end{equation}
Alternatively, inequality \eqref{LogGN1} can be expressed with respect to fractional powers of the sub-Laplacian via   
    \begin{equation}\label{LogGN2}
\int_{\H}\frac{|u|^{2}}{\|u\|^{2}_{L^{2}(\H)}}\log\left(\frac{|u|^{2}}{\|u\|^{2}_{L^{2}(\H)}}\right) dx \leq \frac{q}{q-2}\log\left(C_{B,s}^{a}\|U_{s}\|^{a}_{\rm{op}}\frac{\langle \Lss u,u\rangle^{a}_{L^{2}(\H)}\|u\|^{2(1-a)}_{L^{\sigma}(\H)}}{\|u\|^{2}_{L^{2}(\H)}}\right),
\end{equation}
where the estimate for the norm $\|U_{s}\|_{\text{op}}$ is given by \eqref{est1}.

%In addition, if $\frac{2Q}{Q-2s_{2}}<q\leq \frac{2Q}{Q-2s_{1}}$ with $0\leq s_{2}<s_{1}<1$, we have 
%\begin{multline}\label{LogGN3}
 %   \int_{\H}\frac{|u|^{2}}{\|u\|^{2}_{L^{2}(\H)}}\log\left(\frac{|u|^{2}}{\|u\|^{2}_{L^{2}(\H)}}\right) dx\\
  %  \leq \frac{q}{q-2}\log\left(C\frac{\langle\mathcal{L}_{s_{1}}u,u\rangle_{L^{2}(\H)}^{\frac{Q(q-2)-2s_{2}q}{2(s_{1}-s_{2})}}\langle\mathcal{L}_{s_{2}}u,u\rangle_{L^{2}(\H)}^{\frac{s_{1}(q-2)-Q(q-2)}{2(s_{1}-s_{2})}}}{\|u\|^{2}_{L^{2}(\H)}}\right), 
%\end{multline}
 %where $C>0$ and $u\in H^{s_{1}}(\H)\cap H^{s_{2}}(\H)$.
\end{thm}
%\begin{cor}\label{corGN}
 %   In \eqref{LogGN3}, in the case $s_{2}=0$, $s_{1}=1$ and $\|u\|_{L^{2}(\H)}=1$, we get
  % \begin{equation}\label{LogGN4}
   % \int_{\H}|u|^{2}\log|u| dx
    %\leq \frac{Q}%{4}\log\left(A\langle\mathcal{L}u,u\rangle_{L^{2}(\H)}^{\frac{Q(q-2)}{4}}\right),
   %\end{equation} 
%where 
%\begin{equation}
 %   A=\inf_{q\in(2,\frac{2Q}{Q-2})}C_{GN,1}^{\frac{4}{Q(q-2)}},
%\end{equation}
%where the best constant of the Gagliardo-Nirenberg %inequality $C_{GN,1}$  is defined in \eqref{bestgn}.

%\end{cor}
%\begin{proof}[Proof of Corollary \ref{corGN}]
%By combining $\|u\|_{L^{2}}=1$ and \eqref{LogGN3}, we get
%%    \begin{equation}
 %%       \begin{split}
 %%            \int_{\H}|u|^{2}\log|u|^{2} dx&\leq \frac{q}{q-2}\log\left(C^{\frac{2}{q}}_{GN,1}\langle\mathcal{L}u,u\rangle_{L^{2}(\H)}^{\frac{Q(q-2)}{2}}\right)\\&
%    =\frac{Q}{2}\log\left(C^{\frac{4}{Q(q-%2)}}_{GN,1}\langle\mathcal{L}u,u\rangle_{L^{2}%(\H)}\right).
 %%%       \end{split}
%%    \end{equation}
%    Minimising the constant over the admissible range of %$q$, we obtain \eqref{LogGN4}. 
%\end{proof}
\begin{proof}[Proof of Theorem \ref{thmlogGN}]
From the logarithmic H\"older inequality \eqref{holdernn} we have
\begin{equation*}
    \begin{split}
\int_{\H}\frac{|u(x)|^{2}}{\|u\|^{2}_{L^{2}(\H)}}\log\left(\frac{|u(x)|^{2}}{\|u\|^{2}_{L^{2}(\H)}}\right)dx&\leq \frac{q}{q-2}\log\left(\frac{\|u\|^{2}_{L^{q}(\H)}}{\|u\|^{2}_{L^{2}(\H)}}\right).    
    \end{split}
\end{equation*}

Hence by using \eqref{gnin1}, we have
\begin{equation*}
    \begin{split}
\int_{\H}\frac{|u(x)|^{2}}{\|u\|^{2}_{L^{2}(\H)}}&\log\left(\frac{|u(x)|^{2}}{\|u\|^{2}_{L^{2}(\H)}}\right) dx  \leq  \frac{q}{q-2}\log\left(\frac{\|u\|^{2}_{L^{q}(\H)}}{\|u\|^{2}_{L^{2}(\H)}}\right)\\&
\leq   \frac{q}{q-2}\log\left(C_{B,s}^{a}\frac{\langle \Ls u,u\rangle^{a}_{L^{2}(\H)}\|u\|^{2(1-a)}_{L^{\tau}(\H)}}{\|u\|^{2}_{L^{2}(\H)}}\right).        
    \end{split}
\end{equation*}\label{log.sob.123}
Inequality \eqref{LogGN2} then follows, and the proof of the theorem is complete.
\end{proof}
\section{Hardy-Sobolev inequalities on $\H$ and applications}
The main results in this section are the weighted log-Sobolev inequality on $\H$ and the fractional log-Hardy inequality on $\H$. The next result that was proved in \cite{RT16} is the fractional Hardy inequality on $\H$ with the best constant and  is used for the proof of the aforesaid inequalities.

\begin{thm}[Fractional Hardy inequality]\label{thm.Hardy}
    Let $s\in(0,1]$. We have  the following inequality on $\H$ 
    \begin{equation}\label{eq.Hardy}
        \int_{\H}\frac{|u(x)|^{2}}{|x|^{2s}}dx\leq C_{BH,s}\langle\Lss u,u\rangle_{L^{2}(\H)},
    \end{equation}
   where $|\cdot|$ denotes the Kaplan norm,  and the sharp constant $C_{BH,s}$ is given by  
    \begin{equation}\label{besthar}
        C_{BH,s}:=\|V_{s}\|_{\rm{op}}\frac{\Gamma(1-s)\Gamma^{2}\left(\frac{n}{2}\right)}{2^{2n+3s}\Gamma^{2}\left(\frac{n+s}{2}\right)}
    \end{equation}
    where estimate for the norm of $V_{s}$ is given in \eqref{est2}.
\end{thm}
With the use of Theorem \ref{thm.Hardy} we prove the following generalised version of it which includes it as a special case.
\begin{thm}[Fractional Hardy-Sobolev inequality]\label{FHS}
     Let $s\in(0,1]$ and let $0\leq\beta\leq 2s$. Then  we have 
  
      \begin{equation}\label{eq.Hardy-S}
         \left(\int_{\H}\frac{|u(x)|^{2^{*}_{\beta}}}{|x|^{\beta}}dx\right)^{\frac{2}{2^{*}_{\beta}}}\leq \left(C^{\frac{\beta}{2s}}_{BH,s}\left(C_{B,s}\|U_{s}\|_{\rm{op}}\right)^{\frac{n+1}{n+1-s}\frac{2s-\beta}{2s}}\right)^{\frac{2}{2^{*}_{\beta}}} \langle\Lss u,u\rangle_{L^{2}(\H)},
     \end{equation}
where $2^{*}_{\beta}=\frac{2(Q-\beta)}{Q-2s}$ and $C_{B,s}$ and $C_{BH,s}$ are given by \eqref{bestfrac} and \eqref{besthar}, respectively.
\end{thm}
Indeed observe that the fractional Hardy-Sobolev inequality \eqref{eq.Hardy-S} gives the fractional Hardy inequality \eqref{eq.Hardy} when $\beta=2s$, and the fractional Sobolev \eqref{1b} inequality when $\beta=0$.

\begin{proof}[Proof of Theorem \ref{FHS}]
For $\beta=0$ and $\beta=2s$, we already have the fractional Sobolev and Hardy inequalities in \eqref{1b} and \eqref{eq.Hardy}, respectively. Now for $0<\beta<2s$, using H\"{o}lder's inequality with $\frac{\beta}{2s}+\frac{2s-\beta}{2s}=1$,   we have
    \begin{equation*}
    \begin{split}
        \int_{\H}\frac{|u(x)|^{2^{*}_{\beta}}}{|x|^{\beta}}dx&=\int_{\H}\frac{|u(x)|^{\frac{\beta}{s}}|u(x)|^{2^{*}_{\beta}-\frac{\beta}{s}}}{|x|^{\beta}}dx\\&
\leq \left(\int_{\H}\frac{|u(x)|^{2}}{|x|^{2s}}dx\right)^{\frac{\beta}{2s}}\left(\int_{\H}|u(x)|^{\frac{2Q}{Q-2s}}dx\right)^{\frac{2s-\beta}{2s}}\\&
        \stackrel{\eqref{1b}, \eqref{eq.Hardy}}\leq C^{\frac{\beta}{2s}}_{BH,s}\left(C_{B,s}\|U_{s}\|_{\text{op}}\right)^{\frac{n+1}{n+1-s}\frac{2s-\beta}{2s}} \langle\Lss u,u\rangle^{\frac{2^{*}_{\beta}}{2}}_{L^{2}(\H)},
    \end{split}
    \end{equation*}
    completing the proof.
        %=\left(\|V_{s}\|_{\text{op}}\frac{\Gamma(1-s)\Gamma^{2}\left(\frac{n}{2}\right)}{2^{2n+3s}\Gamma^{2}\left(\frac{n+s}{2}\right)}\right)^{\frac{\beta}{2s}}\left(C_{B,s}\|U_{s}\|_{\text{op}}\right)^{\frac{n+1}{n+1-s}\frac{2s-\beta}{2s}} \langle\Lss u,u\rangle^{\frac{2^{*}_{\beta}}{2}}_{L^{2}(\H)}
\end{proof}

Next, we prove a weighted version of the log-Sobolev inequality in $\H$ that involves fractional powers of the sub-Laplacian $\Lss$. Following similar arguments, the analogous result can be shown for the log-Sobolev inequality involving the modified sub-Laplacian $\Ls$. 
 
\begin{thm}\label{wwsobingrthm}

Let $s\in(0,1]$ and let $0\leq \beta < 2s$. Then we have the following weighted log-Sobolev inequality of fractional order on $\H$
   
      \begin{equation}\label{LogwSobolev2}
  \begin{split}
       &\int_{\H}\frac{|x|^{-\frac{2\beta}{2^{*}_{\beta}}}|u(x)|^{2}}{\||\cdot|^{-\frac{\beta}{2^{*}_{\beta}}}u\|^{2}_{L^{2}(\H)}}\log\left(\frac{|x|^{-\frac{2\beta}{2^{*}_{\beta}}}|u(x)|^{2}}{\||\cdot|^{-\frac{\beta}{2^{*}_{\beta}}}u\|^{2}_{L^{2}(\H)}}\right) dx\\&
       \leq \frac{Q-\beta}{2s-\beta}\log\left(\left(C^{\frac{\beta}{2s}}_{BH,s}(C_{B,s}\|U_{s}\|_{\rm{op}})^{\frac{n+1}{n+1-s}\frac{2s-\beta}{2s}}\right)^{\frac{2}{2^{*}_{\beta}}}\frac{\langle \Lss u,u\rangle_{L^{2}(\H)}}{\||\cdot|^{-\frac{\beta}{2^{*}_{\beta}}}u\|^{2}_{L^{2}(\H)}}\right),
   \end{split}
\end{equation}
  where $2^{*}_{\beta}=\frac{2(Q-\beta)}{Q-2s}$.

\end{thm}

\begin{proof}
Let $q=\frac{2(Q-\beta)}{Q-2s}$. Using the fact that $0\leq \beta <2s$, we have 
\begin{equation*}
    q-2=\frac{2(Q-\beta)}{Q-2s}-2=\frac{2(Q-\beta-Q+2s)}{Q-2s}=\frac{2(2s-\beta)}{Q-2s}> 0.
\end{equation*}
The latter implies that $q> 2=p$, and  we compute
 \begin{equation*}\label{vychetstploghsgr}
     0<\frac{q}{q-2}=\frac{\frac{2(Q-\beta)}{Q-2s}}{\frac{2(Q-\beta)}{Q-2s}-2}
     =\frac{\frac{Q-\beta}{Q-2s}}{\frac{(Q-\beta)}{Q-2s}-1}
     =\frac{Q-\beta}{2s-\beta}.
 \end{equation*}
 %Also, from the last fact we have $ap>\beta$.
 By using the logarithmic H\"{o}lder's inequality as in Lemma \ref{holder}, and Theorem \ref{FHS}, we have
 \begin{equation*}
     \begin{split}
        \int_{\H}&\frac{|x|^{-\frac{2\beta}{2^{*}_{\beta}}}|u(x)|^{2}}{\||\cdot|^{-\frac{\beta}{2^{*}_{\beta}}}u\|^{2}_{L^{2}(\H)}}\log\left(\frac{|x|^{-\frac{2\beta}{2^{*}_{\beta}}}|u(x)|^{2}}{\||\cdot|^{-\frac{\beta}{2^{*}_{\beta}}}u\|^{2}_{L^{2}(\H)}}\right) dx\leq \frac{q}{q-2}\log\left(\frac{\||\cdot|^{-\frac{\beta}{2^{*}_{\beta}}} u\|^{2}_{L^{2^{*}_{\beta}}(\H)}}{\||\cdot|^{-\frac{\beta}{2^{*}_{\beta}}} u\|^{2}_{L^{2}(\H)}}\right)\\&
       \leq \frac{Q-\beta}{2s-\beta}\log\left(\left(C^{\frac{\beta}{2s}}_{BH,s}(C_{B,s}\|U_{s}\|_{\text{op}})^{\frac{n+1}{n+1-s}\frac{2s-\beta}{2s}}\right)^{\frac{2}{2^{*}_{\beta}}}\frac{\langle \Lss u,u\rangle_{L^{2}(\H)}}{\||\cdot|^{-\frac{\beta}{2^{*}_{\beta}}}u\|^{2}_{L^{2}(\H)}}\right),
     \end{split}
 \end{equation*}
completing the proof.
\end{proof}
\begin{rem}
We note that inequality \eqref{LogwSobolev2} for $\beta=0$ gives the logarithmic Sobolev inequality \eqref{LogSobolev2}. However, the logarithmic Hardy inequality cannot be deduced from \eqref{LogwSobolev2} since then we would need to consider the forbidden, by the logarithmic H\"older inequality \eqref{holdernn}, case where  $p=q$. In the next theorem, we develop a different line of arguments to prove the logarithmic Hardy inequality on $\H$. In the Euclidean setting, the logarithmic first order ($s=1$) Hardy inequality was obtained in \cite{DDFT10}.
\end{rem}

\begin{thm}\label{THM:weHarnon}
Let $s\in(0,1]$, be such that $0\leq \beta<2s<Q.$ Then, the log-Hardy inequality on $\H$ reads as follows

     \begin{equation*}
      \begin{split}
          \int_{\H}&\frac{\frac{|u(x)|^{2}}{|x|^{2s-\beta}}}{\left\|\frac{u}{|\cdot|^{\frac{2s-\beta}{2}}}\right\|^{2}_{L^{2}(\H)}}\log\left(\frac{|x|^{(Q-2s)(1-\frac{\beta}{2s-\beta})}|u(x)|^{2}}{\left\|\frac{u}{|\cdot|^{\frac{2s-\beta}{2}}}\right\|^{2}_{L^{2}(\H)}}\right) dx\\&
          \leq \frac{Q-\beta}{2s-\beta}\int_{\H}\frac{|u|^{2}}{|x|^{2s-\beta}}dx\log\left(C_{Lw,s}\frac{\langle\Lss u,u\rangle_{L^{2}(\H)}}{\left\|\frac{u}{|\cdot|^{\frac{2s-\beta}{2}}}\right\|^{2}_{L^{2}(\H)}}\right),
            \end{split}
     \end{equation*}
    where we have set 
    \[
    C_{Lw,s}=\left(C^{\frac{\beta}{2s}}_{BH,s}\left(C_{B,s}\|U_{s}\|_{\text{op}}\right)^{\frac{n+1}{n+1-s}\frac{2s-\beta}{2s}}\right)^{\frac{2}{2^{*}_{\beta}}},\,\,\,2^{*}_{\beta}=\frac{2(Q-\beta)}{Q-2s}.
    \]
   In particular, for all $u$ such that $\int_{\H}\frac{|u|^{2}}{|x|^{2s-\beta}}dx=1$, we have 
   \begin{equation}\label{weloghar1non}
         \int_{\H}\frac{|u(x)|^{2}}{|x|^{2s-\beta}}\log\left(|x|^{(Q-2s)(1-\frac{\beta}{2s-\beta})}|u(x)|^{2}\right) dx \leq \frac{Q-\beta}{2s-\beta}\log\left(C_{Lw,s}\langle\Lss u,u\rangle_{L^{2}(\H)}\right).
     \end{equation}
    
\end{thm}
\begin{proof}
%  Let us choose $q\in(p,p_{\beta}^{*})$, where $\p=\frac{(d-\beta)p}{d-ap}$. Let  $\theta=\frac{(\p-q)p}{\p-p}$, then we have $\theta\in(0,p)$. 
 Let us choose $q\in(2,2_{\beta}^{*})$, where $2^{*}_{\beta}=\frac{2(Q-\beta)}{Q-2s}$. If we set $\theta=\frac{2(\p-q)}{\p-2}$, then we have $\theta\in(0,2)$. Let us make some preparatory computations:
\begin{multline}\label{wevyrtheta}
   \theta=\frac{2(\p-q)}{\p-2}=\frac{2\left(\frac{2(Q-\beta)}{Q-2s}-q\right)}{\frac{2(Q-\beta)}{Q-2s}-2} \\ =\frac{\frac{2}{Q-2s}\left(2(Q-\beta)-q(Q-2s)\right)}{\frac{1}{Q-2s}(2(Q-\beta)-2Q+4s)}=\frac{2(Q-\beta)}{2s-\beta}-\frac{(Q-2s)q}{2s-\beta},
\end{multline}
\begin{equation}\label{wep(qth)}
    2\frac{q-\theta}{2-\theta}=2\frac{q-2\frac{\p-q}{\p-2}}{2-2\frac{\p-q}{\p-2}}=\frac{q(\p-2)-2(\p-q)}{\p-2-\p+q}=\frac{\p(q-2)}{q-2}=\p,
\end{equation}
\begin{equation}\label{weptp}
    \frac{2-\theta}{2}=1-\frac{\theta}{2}=\frac{q-2}{\p-2},
\end{equation}
and 
\begin{equation}\label{hbeta}
    -\frac{\beta \theta}{2}=\beta\frac{2-\theta}{2}-\beta.
\end{equation}
By using H\"{o}lder's inequality \eqref{holder.gen} with $\frac{\theta}{2}+\frac{2-\theta}{2}=1$ and the above calculations we get,
%\textcolor{red}{Are you sure? I am getting $|x|^{-\frac{\theta\beta}{p}+\beta}$ in the denominator in the second line; I also put it in red there. who is mistaken? please check again.}
\begin{equation*}
    \begin{split}
        \int_{\H}\frac{|u(x)|^{q}}{|x|^{Q-\beta-\frac{q}{2}(Q-2s)}}dx &=\int_{\H}\frac{|u(x)|^{\theta}}{|x|^{s\theta }}\frac{|u(x)|^{q-\theta}}{|x|^{Q-\beta-\frac{q}{2}(Q-2s)-\theta s}}dx\\&
        \stackrel{\eqref{wevyrtheta}}=\int_{\H}\frac{|u(x)|^{\theta}}{|x|^{\theta s}}\frac{|u(x)|^{q-\theta}}{|x|^{-\frac{\theta \beta}{2}}}dx\\&
        \stackrel{\eqref{hbeta}}= \int_{\H}\frac{|u(x)|^{\theta}}{|x|^{\theta s}}\frac{|u(x)|^{q-\theta}}{|x|^{\beta\frac{2-\theta}{2}-\beta}}dx\\&
        =\int_{\H}\frac{|u(x)|^{\theta}}{|x|^{\theta s-\beta}}\frac{|u(x)|^{q-\theta}}{|x|^{\beta\frac{2-\theta}{2}}}dx\\&
        \leq \left(\int_{\H}\frac{|u(x)|^{2}}{|x|^{2s-\frac{2\beta }{\theta}}}dx\right)^{\frac{\theta}{2}}\left(\int_{\H}\frac{|u(x)|^{\frac{2(q-\theta)}{2-\theta}}}{|x|^{\beta}}dx\right)^{\frac{2-\theta}{2}}\\&
\stackrel{\eqref{wep(qth)},\eqref{weptp}}=\left(\int_{\H}\frac{|u(x)|^{2}}{|x|^{2s-\beta\frac{\p-2}{\p-q}}}dx\right)^{\frac{\p-q}{\p-2}}\left(\int_{\H}\frac{|u(x)|^{\p}}{|x|^{\beta}}dx\right)^{\frac{q-2}{\p-2}},
    \end{split}
\end{equation*}
also, in the last line we used $\theta=\frac{2(\p-q)}{\p-2}.$
If for $r>0$ we set $q=2+r$ then the above inequalities can be summarised as follows:
\begin{multline}\label{wesrnon}
        \int_{\H}\frac{|u(x)|^{r+2}}{|x|^{Q-\beta-\frac{(r+2)}{2}(Q-2s)}}dx
        \\ \leq \left(\int_{\H}\frac{|u|^{2}}{|x|^{2s-\beta\frac{\p-2}{\p-(r+2)}}}dx\right)^{\frac{\p-(r+2)}{\p-2}}\left(\int_{\H}\frac{|u|^{\p}}{|x|^{\beta}}dx\right)^{\frac{(r+2)-2}{\p-2}}.
\end{multline}
On the other hand, when  $r \rightarrow 0$, inequality \eqref{wesrnon} is the equality
\begin{equation}\label{webrnon}
    \begin{split}
        \int_{\H}\frac{|u|^{2}}{|x|^{Q-\beta-\frac{2}{2}(Q-2s)}}dx&
        =\left(\int_{\H}\frac{|u|^{2}}{|x|^{2s-\beta\frac{\p-2}{\p-2}}}dx\right)^{\frac{\p-2}{\p-2}}\left(\int_{\H}\frac{|u|^{\p}}{|x|^{\beta}}dx\right)^{\frac{2-2}{\p-2}}.
    \end{split}
\end{equation}
Hence using \eqref{wesrnon} and \eqref{webrnon}
%\begin{equation}\label{welogosn1non}
%\begin{split}
 %       &\frac{1}{r}\int_{\G}\left(\frac{|u|^{r+p}}{|x|^{Q-\beta-\frac{(r+p)}{p}(Q-ap)}}-\frac{|u|^{p}}{|x|^{Q-\beta-\frac{p}{p}(Q-ap)}}\right)dx\\&
  %      \leq \frac{1}{r}\Biggl{[}\left(\int_{\G}\frac{|u|^{p}}{|x|^{ap-\beta\frac{\p-p}{\p-(r+p)}}}dx\right)^{\frac{\p-(r+p)}{\p-p}}\left(\int_{\G}\frac{|u|^{\p}}{|x|^{\beta}}dx\right)^{\frac{(r+p)-p}{\p-p}}\\&
   %     -\left(\int_{\G}\frac{|u|^{p}}{|x|^{ap-\beta\frac{\p-p}{\p-p}}}\right)^{\frac{\p-p}{\p-p}}\left(\int_{\G}\frac{|u|^{\p}}{|x|^{\beta}}dx\right)^{\frac{p-p}{\p-p}}\Biggl{]},
%\end{split}
%\end{equation}
we get
\begin{equation}\label{welogosn2non}
\begin{split}
        &\lim_{r\rightarrow 0}\frac{1}{r}\int_{\H}\left(\frac{|u|^{r+2}}{|x|^{Q-\beta-\frac{(r+2)}{2}(Q-2s)}}-\frac{|u|^{2}}{|x|^{Q-\beta-\frac{2}{2}(Q-2s)}}\right)dx\\&
        \leq \lim_{r\rightarrow 0}\frac{1}{r}\Biggl{[}\left(\int_{\H}\frac{|u|^{2}}{|x|^{2s-\beta\frac{\p-2}{\p-(r+2)}}}dx\right)^{\frac{\p-(r+2)}{\p-2}}\left(\int_{\H}\frac{|u|^{\p}}{|x|^{\beta}}dx\right)^{\frac{(r+2)-2}{\p-2}}\\&
        -\left(\int_{\H}\frac{|u|^{2}}{|x|^{2s-\beta\frac{\p-2}{\p-2}}}dx\right)^{\frac{\p-2}{\p-2}}\left(\int_{\H}\frac{|u|^{\p}}{|x|^{\beta}}dx\right)^{\frac{2-2}{\p-2}}\Biggl{]}.
\end{split}
\end{equation}
Let us now compute the left-hand side of the inequality \eqref{welogosn2non} with the use of the l'H\^{o}pital rule in the variable $r$:
\begin{equation}\label{otkrlevchnon}
    \begin{split}
        \lim_{r\rightarrow 0}\frac{1}{r}\left(\frac{|u|^{r+2}}{|x|^{Q-\beta-\frac{(r+2)}{2}(Q-2s)}}-\frac{|u|^{2}}{|x|^{Q-\beta-\frac{2}{2}(Q-2s)}}\right)&=\lim_{r\rightarrow 0}\frac{d}{dr}\frac{|u|^{r+2}}{|x|^{Q-\beta-\frac{(r+2)}{2}(Q-2s)}}\\&
        =\frac{|u|^{2}\log(|u|^{2}|x|^{Q-2s})}{2|x|^{2s-\beta}}.
    \end{split}
\end{equation}
For the right-hand side of  \eqref{welogosn2non}, we define the auxiliary function $z(r):=(f(r))^{g(r)}$, where
\begin{equation*}
    f(r)=\int_{\H}\frac{|u|^{2}}{|x|^{2s-\beta\frac{\p-2}{\p-(r+2)}}}dx,
\end{equation*}
and 
\begin{equation*}
    g(r)=\frac{\p-(r+2)}{\p-2}.
\end{equation*}
The derivatives of the functions $f,g$ with respect to the variable $r$ are as follows
\begin{equation}\label{derfrnon}
    \begin{split}
       \frac{df(r)}{dr}&=\frac{d}{dr}\int_{\H}\frac{|u|^{2}}{|x|^{2s-\beta\frac{\p-2}{\p-(r+2)}}}dx\\&
       =\frac{\beta(\p-2)}{(\p-(r+2))^{2}}\int_{\H}\frac{|u|^{2}}{|x|^{2s-\beta\frac{\p-2}{\p-(r+2)}}}\log(|x|)dx,
    \end{split}
\end{equation}
and 
\begin{equation}\label{dergrnon}
    \frac{dg(r)}{dr}=\frac{d}{dr}\frac{\p-(r+2)}{\p-2}=-\frac{1}{\p-2}.
\end{equation}
Hence the derivative of $z(r)$ is given by 
\begin{equation}\label{compzrnon}
\begin{split}
    \frac{dz(r)}{dr}&=(f(r))^{g(r)}\left(\frac{dg(r)}{dr}\log(f(r))+\frac{g(r)\frac{df(r)}{dr}}{f(r)}\right)\\&
    \stackrel{\eqref{derfrnon},\eqref{dergrnon}}=z(r)\Biggl{(}-\frac{1}{\p-2}\log\left(\int_{\H}\frac{|u|^{2}}{|x|^{2s-\beta\frac{\p-2}{\p-(r+2)}}}dx\right)\\&
    +\frac{\beta}{\p-(r+2)}\frac{\int_{\H}\frac{|u|^{2}}{|x|^{2s-\beta\frac{\p-2}{\p-(r+2)}}}\log(|x|)dx}{\int_{\H}\frac{|u|^{2}}{|x|^{2s-\beta\frac{\p-2}{\p-(r+2)}}}dx}\Biggl{)}.
\end{split}
\end{equation}
We also have
\begin{equation}\label{compharsobnon}
    \begin{split}
      \frac{d}{dr}\left(\int_{\H}\frac{|u|^{\p}}{|x|^{\beta}}dx\right)^{\frac{(r+2)-2}{\p-2}}&=  \frac{1}{\p-2}\left(\int_{\H}\frac{|u|^{\p}}{|x|^{\beta}}dx\right)^{\frac{(r+2)-2}{\p-2}}\\&
      \times\log\left(\int_{\H}\frac{|u|^{\p}}{|x|^{\beta}}dx\right).
    \end{split}
\end{equation}
Using the above derivatives and l'H\^{o}pital's rule, we have
\begin{equation*}
    \begin{split}
        &\lim_{r\rightarrow 0}\frac{1}{r}\Biggl{[}\left(\int_{\H}\frac{|u|^{2}}{|x|^{2s-\beta\frac{\p-2}{\p-(r+2)}}}dx\right)^{\frac{\p-(r+2)}{\p-2}}\left(\int_{\H}\frac{|u|^{\p}}{|x|^{\beta}}dx\right)^{\frac{(r+2)-2}{\p-2}}\\&
        -\left(\int_{\H}\frac{|u|^{2}}{|x|^{2s-\beta\frac{\p-2}{\p-2}}}dx\right)^{\frac{\p-2}{\p-2}}\left(\int_{\H}\frac{|u|^{\p}}{|x|^{\beta}}dx\right)^{\frac{p-p}{\p-p}}\Biggl{]}\\&
        =\lim_{r\rightarrow 0}\frac{d}{dr}\left[\left(\int_{\H}\frac{|u|^{2}}{|x|^{2s-\beta\frac{\p-2}{\p-(r+2)}}}dx\right)^{\frac{\p-(r+2)}{\p-2}}\left(\int_{\H}\frac{|u|^{\p}}{|x|^{\beta}}dx\right)^{\frac{(r+2)-2}{\p-2}}\right]
        \end{split}
\end{equation*}
        \begin{equation*}
    \begin{split}
        &
        =\lim_{r\rightarrow 0}\frac{d}{dr}\left(z(r)\left(\int_{\H}\frac{|u|^{\p}}{|x|^{\beta}}dx\right)^{\frac{(r+2)-2}{\p-2}}\right)\\&
        =\lim_{r\rightarrow 0}\left[\frac{dz(r)}{dr}\left(\int_{\H}\frac{|u|^{\p}}{|x|^{\beta}}dx\right)^{\frac{(r+2)-2}{\p-2}}+z(r)\frac{d}{dr}\left(\int_{\H}\frac{|u|^{\p}}{|x|^{\beta}}dx\right)^{\frac{(r+2)-2}{\p-2}}\right]\\&
    \stackrel{\eqref{compzrnon},\eqref{compharsobnon}}=\frac{\int_{\H}\frac{|u|^{2}}{|x|^{2s-\beta}}dx}{\p-2}\Biggl{[}-\log\left(\int_{\H}\frac{|u|^{2}}{|x|^{2s-\beta}}dx\right)\\&
    +\beta\frac{\int_{\H}\frac{|u|^{2}}{|x|^{2s-\beta}}\log(|x|)dx}{\int_{\H}\frac{|u|^{2}}{|x|^{2s-\beta}}dx}+\log\left(\int_{\H}\frac{|u|^{\p}}{|x|^{\beta}}dx\right)\Biggl{]}\\&
    =\frac{\int_{\H}\frac{|u|^{2}}{|x|^{2s-\beta}}dx}{\p-2}\Biggl{[}\beta\frac{\int_{\H}\frac{|u|^{2}}{|x|^{2s-\beta}}\log(|x|)dx}{\int_{\H}\frac{|u|^{2}}{|x|^{2s-\beta}}dx}+\log\frac{\left(\int_{\H}\frac{|u|^{\p}}{|x|^{\beta}}dx\right)}{\left(\int_{\H}\frac{|u|^{2}}{|x|^{2s-\beta}}dx\right)}\Biggl{]}\\&
    =\frac{I_{1}}{\p-2}\left(\beta\frac{I_{3}}{I_{1}}+\log\frac{I_{2}}{I_{1}}\right),
\end{split}
\end{equation*}
where $I_1$, $I_2$ and $I_3$ stand for the next integrals
\begin{equation*}
    I_{1}=\int_{\H}\frac{|u|^{2}}{|x|^{2s-\beta}}dx,
\end{equation*}
\begin{equation*}
    I_{2}=\int_{\H}\frac{|u|^{\p}}{|x|^{\beta}}dx,
\end{equation*}
and 
\begin{equation*}
    I_{3}=\int_{\H}\frac{|u|^{2}}{|x|^{2s-\beta}}\log(|x|)dx.
\end{equation*}
Therefore, we also have
\begin{equation*}
    \begin{split}
        &\lim_{r\rightarrow 0}\frac{1}{r}\Biggl{[}\left(\int_{\H}\frac{|u|^{2}}{|x|^{2s-\beta\frac{\p-2}{\p-(r+2)}}}dx\right)^{\frac{\p-(r+2)}{\p-2}}\left(\int_{\H}\frac{|u|^{\p}}{|x|^{\beta}}dx\right)^{\frac{(r+2)-2}{\p-2}}\\&
        -\left(\int_{\H}\frac{|u|^{2}}{|x|^{2s-\beta\frac{\p-2}{\p-2}}}dx\right)^{\frac{\p-2}{\p-2}}\left(\int_{\H}\frac{|u|^{\p}}{|x|^{\beta}}dx\right)^{\frac{2-2}{\p-2}}\Biggl{]}
    \end{split}
\end{equation*}
        \begin{equation}\label{otkrpravchnon}
    \begin{split}
        &=\frac{I_{1}}{\p-2}\left(\beta\frac{I_{3}}{I_{1}}+\log\frac{I_{2}}{I_{1}}\right)\\&
        =\frac{\beta}{\p-2}I_{3}+\frac{I_{1}}{\p-2}\frac{2\p }{2\p }\log\frac{I_{2}}{I_{1}}\\&
        =\frac{\beta}{\p-2}I_{3}+\frac{I_{1}\p}{2(\p-2)}\log\frac{I^{\frac{2}{\p }}_{2}}{I_{1}^{1-1+\frac{2}{\p}}}\\&
        =\frac{\beta}{\p-2}I_{3}+\frac{I_{1}\p}{2(\p-2)}\log\frac{I^{\frac{2}{\p }}_{2}}{I_{1}}-\frac{I_{1}\p }{2(\p-2)}\log I_{1}^{-1+\frac{2}{\p}}\\&
        =\frac{\beta}{\p-2}I_{3}+\frac{I_{1}\p}{2(\p-2)}\log\frac{I^{\frac{2}{\p }}_{2}}{I_{1}}+\frac{I_{1}}{2}\log I_{1}.
    \end{split}
\end{equation}

Using the expression \eqref{otkrlevchnon} for the left-hand side of  \eqref{welogosn2non} and the expression \eqref{otkrpravchnon} for the right-hand side of it we get
\begin{equation}\label{predposnon}
\begin{split}
     \int_{\H}\frac{|u|^{2}\log(|u|^{2}|x|^{Q-2s})}{|x|^{2s-\beta}}&dx\leq \frac{2\beta }{\p-2}I_{3}+\frac{I_{1}\p}{\p-2}\log\frac{I^{\frac{2}{\p}}_{2}}{I_{1}}+I_{1}\log I_{1}\\&
     =\int_{\H}\frac{|u|^{2}}{|x|^{2s-\beta}}\log\left(|x|^{\frac{\beta(Q-2s)}{2s-\beta}}\int_{\H}\frac{|u|^{2}}{|y|^{2s-\beta}}dy\right)dx\\&
     +\frac{(Q-\beta)\int_{\H}\frac{|u|^{2}}{|x|^{2s-\beta}}dx}{2s-\beta}\log \left(\frac{\left(\int_{\H}\frac{|u|^{\p}}{|x|^{\beta}}dx\right)^{\frac{2}{\p}}}{\int_{\H}\frac{|u|^{2}}{|x|^{2s-\beta}}dx}\right)\,,
\end{split}
\end{equation}
since $\frac{\p}{2(\p-2)}=\frac{Q-\beta}{2s-\beta}$, and so 
\[
\frac{I_{1}\p}{2(\p-2)}\log\frac{I^{\frac{2}{\p }}_{2}}{I_{1}}=\frac{(Q-\beta)}{2s-\beta}\int_{\H}\frac{|u|^{2}}{|x|^{2s-\beta}}dx \log \left(\frac{\left(\int_{\H}\frac{|u|^{\p}}{|x|^{\beta}}dx\right)^{\frac{2}{\p}}}{\int_{\H}\frac{|u|^{2}}{|x|^{2s-\beta}}dx}\right)\,,
\]
while also, since $\frac{\beta }{\p-2}= \frac{\beta(Q-2s)}{2s-\beta}$, it is easy to check that 
\[
\frac{2\beta }{\p-2}I_{3}+I_{1}\log I_{1}=\int_{\H}\frac{|u|^{2}}{|x|^{2s-\beta}}\log\left(|x|^{\frac{\beta(Q-2s)}{2s-\beta}}\int_{\H}\frac{|u|^{2}}{|x|^{2s-\beta}}dx\right)dx\,.
\]

From Theorem \ref{FHS}, if $0\leq \beta<2s<Q$, we have the Hardy-Sobolev inequality with $\p=\frac{2(Q-\beta)}{Q-2s}$ in the following form:
\begin{equation*}
    \left\|\frac{u}{|x|^{\frac{\beta}{\p}}}\right\|_{L^{\p}(\H)}\leq C^{\frac{1}{2}}_{Lw,s}\langle\Lss u,u\rangle^{\frac{1}{2}}_{L^{2}(\H)}\,,
\end{equation*}
with $C_{Lw,s}=\left(C^{\frac{\beta}{2s}}_{BH,s}\left(C_{B,s}\|U_{s}\|_{\text{op}}\right)^{\frac{n+1}{n+1-s}\frac{2s-\beta}{2s}}\right)^{\frac{2}{2^{*}_{\beta}}}$.
Finally, by using this in \eqref{predposnon}, we arrive at 
\begin{equation*}
\begin{split}
    &\int_{\H}\frac{\frac{|u|^{2}}{|x|^{2s-\beta}}}{\int_{\H}\frac{|u|^{2}}{|x|^{2s-\beta}}dx}\log\left(\frac{|u|^{2}|x|^{(Q-2s)(1-\frac{\beta}{2s-\beta})}}{\int_{\H}\frac{|u|^{2}}{|x|^{2s-\beta}}dx}\right)dx \\ &\leq\frac{Q-\beta}{2s-\beta} \int_{\H}\frac{|u|^{2}}{|x|^{2s-\beta}}dx\log\frac{\left(\int_{\H}\frac{|u|^{\p}}{|x|^{\beta}}dx\right)^{\frac{2}{\p}}}{\int_{\H}\frac{|u|^{2}}{|x|^{2s-\beta}}dx}\\&
    \leq\frac{Q-\beta}{2s-\beta}\int_{\H}\frac{|u|^{2}}{|x|^{2s-\beta}}dx\log\left(C_{Lw,s}\frac{\langle\Lss u,u\rangle_{L^{2}(\H)}}{\int_{\H}\frac{|u|^{2}}{|x|^{2s-\beta}}dx}\right),
\end{split}
\end{equation*}
where for the first inequality we have subtracted the the first term of the right-hand side of \eqref{predposnon} from the left-hand side of \eqref{predposnon}, and the proof is complete. 
\end{proof}

\begin{rem}\label{REM:CCR-nw}
By taking $\beta=s$ in \eqref{weloghar1non}, we obtain an interesting inequality
   \begin{equation}\label{weloghar1non-nw}
         \int_{\H}\frac{|u(x)|^{2}}{|x|^{s}}\log\left(|u(x)|\right) dx \leq \frac{Q-s}{2s}\log\left(C_{Lw,s}\langle\Lss u,u\rangle_{L^{2}(\H)}\right),
     \end{equation}
    for all $u$ such that $\int_{\H}\frac{|u|^{2}}{|x|^{s}}dx=1$. Inequality \eqref{weloghar1non-nw} is  a combination of radial and logarithmic weights in the usual Hardy inequality.
\end{rem}

\section{Nash inequality on $\H$ and applications}

On the Euclidean space it is proved, see \cite{Bec98}, that the Nash inequality is equivalent to the $L^2$-log-Sobolev inequality. In this section we show that the $L^2$-log-Sobolev inequality on $\H$ implies the Nash inequality on $\H$ with an explicit constant.

\begin{thm}\label{THM:Nash}
 Let $0< s\leq 1.$ Then, Nash's inequality on $\H$ with respect to the (modified) fractional sub-Laplacian reads as follows 
\begin{equation}\label{Nash.g}
 \|u\|_{L^2(\H)}^{2+\frac{4s}{Q}} \leq C_{B,s} \|u\|_{L^1(\H)}^{\frac{4s}{Q}}\langle\Ls u,u\rangle_{L^{2}(\H)}\,.
 \end{equation}
 Alternatively, Nash's inequality on $\H$ with respect to fractional powers of the sub-Laplacian is given by 
\begin{equation}\label{Nash.g1}
 \|u\|_{L^2(\H)}^{2+\frac{4s}{Q}} \leq C_{B,s}\|U\|_{\text{op}} \|u\|_{L^1(\H)}^{\frac{4s}{Q}}\langle\Lss u,u\rangle_{L^{2}(\H)}\,,
 \end{equation}
where $C_{B,s}$ is defined in \eqref{bestfrac}.
\end{thm}
\begin{proof}
Using Jensen's inequality for the function $\theta(u)=\log \frac{1}{u}$
and for the probability measure $|u|^2 \|u\|^{-2}_{L^2(\H)}\,dx$ we get 
 \begin{eqnarray}\label{Nash1}
\log \left(\frac{\|u\|^{2}_{L^2(\H)}}{\|u\|_{L^1(\H)}} \right) & = & \theta \left( \frac{\|u\|_{L^1(\H)}}{\|u\|^{2}_{L^2(\H)}}\right)=\theta \left( \int_{\H}\frac{1}{|u|}|u|^2 \|u\|^{-2}_{L^2(\H)}\,dx\right)\nonumber\\
& \leq & \int_{\H} \theta \left( \frac{1}{|u|}\right) |u|^2 \|u\|^{-2}_{L^2(\H)}\,dx=\int_{\H} \frac{|u|^2}{\|u\|^{2}_{L^2(\H)}}\log |u| \,dx\,.
 \end{eqnarray}
An application of the log-Sobolev inequality as in Theorem \ref{thm:Logsob} yields
 \begin{equation}\label{Nash2}
 \int_{\H} \frac{|u|^2}{\|u\|^{2}_{L^2(\H)}}\log |u| \,dx \leq \log \|u\|_{L^2(\H)}+\frac{Q}{4s} \log \left( C_{B,s} \frac{\langle\Ls u, u\rangle_{L^{2}(\H)}}{\|u\|^{2}_{L^2(\H)}}\right)\,. 
 \end{equation}
  Now using the properties of the logarithm, summing up  inequalities \eqref{Nash1} and \eqref{Nash2}, and rearranging the terms of this sum, we get 
 \[
 \left(1+\frac{Q}{2s} \right) \log \|u\|_{L^2(\H)} \leq \frac{Q}{4s} \log \left(C_{B,s} \|u\|_{L^1(\H)}^{\frac{4s}{Q}}\langle\Ls u, u\rangle_{L^{2}(\H)} \right)\,,
 \]
 which is equivalent to \eqref{Nash.g}. Now, \eqref{Nash.g} implies, as before, \eqref{Nash.g1} and the proof is complete.
\end{proof}
\begin{cor}
For $C_{B,1}$ as in \eqref{bestintsob} we get the following standard form of Nash's inequality on $\H$
\begin{equation}\label{Nash.hor}
 \|u\|_{L^2(\H)}^{2+\frac{4}{Q}} \leq C_{B,1} \|u\|_{L^1(\H)}^{\frac{4}{Q}}\|\nabla_{\H}u\|^{2}_{L^{2}(\H)}\,.
\end{equation}
\end{cor}
\begin{proof}
    The proof follows immediately from \eqref{Nash.g} for $s=1$ and the equality \[\langle\L u,u\rangle^\frac{1}{2}_{L^{2}(\H)}=\|\L^{1/2} u\|_{L^2(\H)}=\||\nabla_{\H}| u\|_{L^2(\H)}=\|\nabla_{\H}u\|_{L^2(\H)}\,.\]
    The proof is complete.
\end{proof}

Standard applications of Nash's inequality relate to Markovian semigroups, see e.g. \cite[Section II.5]{VSCC93}. Below we give a short application of Nash's inequality to estimate the time-decay rate for the solution to the heat equation on $\H$ with respect to $\L$.

\begin{cor}\label{cor:par}
Let $f_0 \geq 0$ be such that $f_0\in L^1(\H)\cap L^2(\H)$.
Then the solution $f$  to the linear heat equation on $\H$ 
\begin{equation}
    \label{heat.eq}
    \partial_t f+\mathcal{L} f=0,\quad f(0,x)=f_0(x),
\end{equation}
satisfies the following time-decay estimate for all $t \geq 0$
\begin{equation*}
    \label{heatest}
    \|f(t,\cdot)\|_{L^2(\H)}\leq \left( \|f_0\|_{L^2(\H)}^{-\frac{4}{Q}}
    +\frac{4}{QC_{B,1}}\|f_0\|_{L^1(\H)}^{-\frac{4}{Q}}t\right)^{-\frac{Q}{4}}\,,
\end{equation*}
 where $C_{B,1}$ is given in \eqref{bestintsob}.
\end{cor}
\begin{proof}
Since  $f_0\geq 0$, the positivity of the heat kernel (see e.g. \cite[p. 48]{VSCC93}) $h_t$ implies that  $f(t,x)\geq 0$, as well. Now, the mass conservation implies the equality of the $L^1$-norms
\begin{equation}\label{L1norm}
\int_{\H} f(t,x) dx=\int_{\H}\int_{\H} h_t(x y^{-1})f_0(y) dydx=\int_{\H} f_0(y) dy,    
\end{equation}
where we applied Fubini's theorem and the fact that  $\|h_t\|_{L^1(\H)}=1.$ Now, multiplying by $f$ the heat equation \eqref{heat.eq}, and integrating over $\H$ we get 
\[
\langle \partial_t f(t,\cdot), f(t,\cdot) \rangle_{L^2(\H)}+\langle \mathcal{L} f(t,\cdot), f(t,\cdot) \rangle_{L^2(\H)}=0\,,
\]
or, equivalently,
\[
\frac{d}{dt}\|f(t,\cdot)\|^{2}_{L^2(\H)}=-2\|\nabla_{\H} f(t,\cdot)\|^{2}_{L^2(\H)}\,.
\]
If we denote by $y(t)$ the norm $y(t):=\|f(t,\cdot)\|^{2}_{L^2(\H)}$, by \eqref{Nash.hor} we get
\begin{equation*}
    \begin{split}
   y'(t)&=\frac{d}{dt}\|f(t,\cdot)\|^{2}_{L^2(\H)}\\&
   =-2\|\nabla_{\H} f(t,\cdot)\|^{2}_{L^2(\H)}\\&
   \stackrel{\eqref{Nash.hor}}\leq -2C^{-1}_{B,1}\|f(t,\cdot)\|^{2+\frac{4}{Q}}_{L^{2}(\H)}\|f(t,\cdot)\|^{\frac{4}{Q}}_{L^{1}(\H)}\\&
\stackrel{\eqref{L1norm}}=-2C^{-1}_{B,1}y^{1+\frac{2}{Q}}(t)\|f_{0}\|^{\frac{4}{Q}}_{L^{1}(\H)},
    \end{split}
\end{equation*}
which, after integration on $t \geq 0$, gives the estimate 
\[
\|f(t,\cdot)\|_{L^2(\H)}\leq \left(\|f_0\|_{L^2(\H)}^{-\frac{4}{Q}}+\frac{4}{QC_{B,1}}\|f_0\|_{L^1(\H)}^{-\frac{Q}{4}}t \right)^{-\frac{Q}{4}}\,,
\]
and this finishes the proof. 
\end{proof}
\section{Poincar\'e inequality on $\H$}
 In the section, we prove several versions of  the Poincar\'e inequality on $\H$. In 1989 Beckner \cite{Bec89} proved the following generalised Poincar\'e inequality on $\mathbb{R}^n$:
\begin{equation}\label{Poi.bec}
\frac{1}{2-p}\left[\int_{\Rn}f^2\,d\mu-\left(\int_{\Rn}|f|^p\,d\mu \right)^{\frac{2}{p}} \right]\leq \int_{\Rn}|\nabla f|^2\,d\mu\,, \quad p \geq 1\,,
\end{equation}
for $d\mu$ the Gaussian measure on $\mathbb{R}^n$. Recall the classical Poincar\'e inequality on $\Rn$:
\begin{equation}
    \label{Poi}
    \int_{\Rn}|f-\mu(f)|^{q}\,d\mu \leq c_o \mu(|\nabla f|^q)\,, \quad q \geq 1\,,
\end{equation}
where we have denoted $\mu(f):=\int_{\Rn}f\,d\mu$.
 It is easy to check that the left-hand side of the generalised Poincar\'e inequality \eqref{Poi.bec} for $p=1$ equals the left-hand side of the classical Poincar\'e inequality \eqref{Poi} for $q=2$. 
 This is true in more generality; indeed if $\nu$ is a probability measure on any measure space $\mathbb{X}$, then we have 
\[
\int_{\mathbb{X}}f^2\,d\nu-\left(\int_{\mathbb{X}}f\,dx\right)^2=\int_{\mathbb{X}}(f-\nu(f))^2\,d\nu\,.
\]
\begin{thm}\label{thm.poi.H}
    Let $\mu$ be the semi-Gaussian measure as in Theorem \ref{semi-g}.  Then, for $p\leq 2$, the  following generalised Poincar\'e inequality with respect to $\mu$ holds true on $\H$:
    \begin{equation}
        \label{Poin.H}
         \int_{\H}|g|^2\,d\mu-\left( \int_{\H}|g|^{p}\,d\mu\right)^{\frac{2}{p}} \leq \frac{2(2-p)}{p}\int_{\H}|\nabla_{\H}g|^2\,d\mu\,.
    \end{equation}
    In particular, for $p=1$ we have 
    \[
    \int_{\H}|g|^2\,d\mu -\left(\int_{\H}|g|\,d\mu \right)^2 \leq 2\int_{\H}|\nabla_{\H}g|^2\,d\mu\,.
    \]
\end{thm}
\begin{proof}
   For some suitable $g$ we consider the function $b=b(q):=q \log \left( \int_{\H}|g|^{\frac{2}{q}}\,d\mu \right)$ for $q> 1$. We denote by $I$ the integral 
   \[
   I(q)=\int_{\H} |g|^{\frac{2}{q}}\,d\mu\,,
   \]
   and hence, we rewrite $b(q)$ as
   \begin{equation*}
       b(q)=q\log I(q).
   \end{equation*}
   Let us compute derivatives of $I(q)$. Then, we have
   \begin{equation*}
       I'(q)=\int_{\H}\log |g||g|^{\frac{2}{q}}\left(-\frac{2}{q^2}\right)\,d\mu,
   \end{equation*}
   and 
    \begin{equation*}
       I''(q)=\int_{\H}\left((\log|g|)^{2}|g|^{\frac{2}{q}}\left(-\frac{2}{q^2}\right)^{2}+\log |g||g|^{\frac{2}{q}}\left(\frac{4}{q^3}\right)\right)\,d\mu.
   \end{equation*}
   Then, by using the last facts, we have
   \[
   b'(q)=\log I(q)+\frac{q I'(q)}{I(q)}\,,
   \]
   and 
   \begin{equation}
   \begin{split}
       b''(q)&=\frac{I(q)(2I'(q)+qI''(q))-q(I'(q))^{2}}{I^{2}(q)}\\&
       =\frac{q}{I^{2}(q)}\Biggl{[}\left(\int_{\H}|g|^{\frac{2}{q}}d\mu\right)\left(\int_{\H}(\log|g|)^{2}|g|^{\frac{2}{q}}\left(-\frac{2}{q^2}\right)^{2}d\mu\right)\\&
       -\left(\int_{\H}\log |g||g|^{\frac{2}{q}}\left(-\frac{2}{q^2}\right)\,d\mu\right)^{2}\Biggl{]}.
   \end{split}
   \end{equation}
   It is then clear that $b$ is a convex function since by the Cauchy–Schwarz inequality we have 
   \[
   \left(\int_{\H}(\log |g|)^2 |g|^{\frac{2}{q}}\left(-\frac{2}{q^2}\right)^2\,d\mu \right)\left( \int_{\H}|g|^{\frac{2}{q}}\,d\mu\right) \geq \left(\int_{\H}\log |g||g|^{\frac{2}{q}}\left(-\frac{2}{q^2}\right)\,d\mu \right)^2\,, 
   \]
   which implies that $b'' \geq 0$. The convexity of $b$ implies in turn the convexity of the function $q \mapsto e^{b(q)}$. Now let us recall the following characterisation of the convexity of functions of one variable: the function $f$ is convex if and only if the function 
   \[
   R(x,y)=\frac{f(y)-f(x)}{y-x}
   \]
   is monotonically non-decreasing in $x$ for fixed $y$, and vice versa. Hence by the latter characterisation we have that the function 
   \[
   \phi(q):= \frac{e^{b(1)}-e^{b(q)}}{1-q}\,,
   \]
   is monotonically non-increasing for $q \geq 1$. Consequently, we get 
   \begin{equation}\label{phi.decreasing}
   \phi(q)\leq \lim_{q \rightarrow 1}\phi(q)\,.
   \end{equation}
   We note that
   \begin{equation}\label{phi(q)}
   \phi(q)=\frac{\int_{\H}|g|^2\,d\mu-\left( \int_{\H}|g|^{\frac{2}{q}}\,d\mu\right)^q}{1-q}\,.
   \end{equation}
   For the computation of the limit $\lim\limits_{q \rightarrow 1}\phi(q)$ we will use the l'H\^{o}pital rule.
   We have 
   \begin{eqnarray*}
       \frac{d}{dq}\left(\int_{\H}|g|^{\frac{2}{q}}\,d\mu \right)^q & = &
       \left(\int_{\H}|g|^{\frac{2}{q}}\,d\mu \right)^q \frac{d}{dq} \left[q \log \left(\int_{\H}|g|^{2/q}\,d\mu \right) \right] \\
      & = &  \left(\int_{\H}|g|^{\frac{2}{q}}\,d\mu \right)^q \left[\log \left(\int_{\H}|g|^{2/q}\,d\mu \right)-\frac{2}{q^2} \log \int_{\H}|g|\,d\mu  \right],
   \end{eqnarray*}
   which implies 
   \begin{eqnarray*}
     \lim_{q \rightarrow 1}\frac{d}{dq}\left(\int_{\H}|g|^{\frac{2}{q}}\,d\mu \right)^q & = & 2 \int_{\H}\ |g|^2 \log|g|\,d\mu
   \end{eqnarray*}
   and so $\lim\limits_{q \rightarrow 1}\phi(q)=2 \int_{\H}\ |g|^2 \log|g|\,d\mu$. Combining this together with \eqref{phi.decreasing} and \eqref{phi(q)} we get 
   \begin{equation*}
       \int_{\H}|g|^2\,d\mu-\left( \int_{\H}|g|^{\frac{2}{q}}\,d\mu\right)^q \leq 2(q-1) \int_{\H}|g|^2 \log |g|\,d\mu\,,
   \end{equation*}
   or after replacing $g$ by $\frac{g}{\|g\|_{L^{2}(\H,\mu)}}$ 
   \begin{equation}\label{before.P}
   \begin{split}
      &\int_{\H}\frac{|g|^2}{\gnorm^2}\,d\mu-\left(\int_{\H}\left(\frac{|g|}{\gnorm}\right)^{\frac{2}{q}}\,d\mu\right)^{q}
      \\&\leq 2 (q-1)\int_{\H}\frac{|g|^2}{\gnorm^2} \log \left(\frac{|g|}{\gnorm} \right)\,d\mu\,. 
   \end{split}
   \end{equation}
   On the other hand, the semi-Gaussian inequality \eqref{gaus.log.sob} gives 
   \begin{equation}
       \label{semi=g-rescaled}
       \int_{\H}\frac{|g|^2}{\gnorm^2}\log \frac{|g|}{\gnorm}\,d\mu \leq \frac{1}{\gnorm^2}\int_{\H}|\nabla_{\H}g|^2\,d\mu\,.
   \end{equation}
   Finally, a combination of \eqref{before.P} together with \eqref{semi=g-rescaled} gives 
  \begin{equation}\label{D+semig}
  \begin{split}
      & \int_{\H}\frac{|g|^2}{\gnorm^2}\,d\mu-\left(\int_{\H}\left(\frac{|g|}{\gnorm}\right)^{\frac{2}{q}}\,d\mu\right)^q\\&
      \leq 2 (q-1) \frac{1}{\gnorm^2}\int_{\H}|\nabla_{\H}g|^2\,d\mu\,,
  \end{split}
  \end{equation}
  or after simplification 
  \begin{equation}\label{forq}
  \int_{\H}|g|^2d\mu-\left(\int_{\H}|g|^{\frac{2}{q}}\right)^q\leq 2 (q-1)\int_{\H}|\nabla_{\H}g|^2\,d\mu \,,
  \end{equation}
   and the proof is complete if one sets  $p=\frac{2}{q}$ in \eqref{forq}.
\end{proof}

\begin{thm}\label{thm.poi.haar}
   For $q > 1$ the following Poincar\'e type inequality with respect to the Haar measure on $\H$ holds true
    \begin{equation}
        \label{Poi.haar.a}
         \int_{\H}\frac{|g|^2}{\|g\|_{L^2(\H)}^2}\,dx-\left( \int_{\H}\left(\frac{|g|}{\|g\|_{L^2(\H)}} \right)^{\frac{2}{q}}\,dx\right)^q \leq \frac{Q(q-1)}{2} \log \left( C_{B,1} \frac{\|\nabla_{\H}g\|_{L^2(\H)}^{2}}{\|g\|^{2}_{L^2(\H)}}\right)\,, 
    \end{equation}
    where $C_{B,1}$ is given in \eqref{bestintsob}.
    Additionally, for $s \in (0,1]$ and for $q \geq 1$, we have the following Poincar\'e type (modified) fractional inequality  
    \begin{equation}
        \label{Poi.haar.b}
         \int_{\H}\frac{|g|^2}{\|g\|_{L^2(\H)}^2}\,dx-\left( \int_{\H}\left(\frac{|g|}{\|g\|_{L^2(\H)}} \right)^{\frac{2}{q}}\,dx\right)^q \leq \frac{Q(q-1)}{2s} \log \left( C_{B,s} \frac{\langle \mathcal{L}_sg,g \rangle_{L^2(\H)}}{\|g\|_{L^2(\H)}^{2}}\right)\,,
    \end{equation}
    where the constant $C_{B,s}$ is as in \eqref{bestfrac}, as well as the following Poincar\'e type inequality that involves fractional powers of the sub-Laplacian $\mathcal{L}$
     \begin{equation}
        \label{Poi.haar.c}
         \int_{\H}\frac{|g|^2}{\|g\|_{L^2(\H)}^2}\,dx-\left( \int_{\H}\left(\frac{|g|}{\|g\|_{L^2(\H)}} \right)^{\frac{2}{q}}\,dx\right)^q \leq \frac{Q(q-1)}{2s} \log \left( C_{B,s}\|U_s\|_{\text{op}} \frac{\langle \mathcal{L}^s g,g \rangle_{L^2(\H)}}{\|g\|_{L^2(\H)}^{2}}\right)\,,
    \end{equation}where the operator norm $\|U_s\|_{\text{op}}$ has been estimated in \eqref{est1}.
    
\end{thm}
\begin{proof}
    Recall that in the proof of Theorem \ref{thm.poi.H} the inequality \eqref{before.P} is obtained independently of the choice  of the considered measure. Hence, with respect to the Haar measure on $\H$ we have 
    \begin{equation*}
         \int_{\H}|g|^2\,dx-\left( \int_{\H}|g|^{\frac{2}{q}
        }\,dx\right)^q \leq 2(q-1) \int_{\H}|g|^2 \log |g|\,dx\,,
    \end{equation*}
    or 
    \begin{equation}
         \label{Poin.haar.1}
         \int_{\H}\frac{|g|^2}{\|g\|_{L^2(\H)}^2}\,dx-\left(\int_{\H}\left(\frac{|g|}{\|g\|_{L^2(\H)}}\right)^{\frac{2}{q}}\,dx\right)^q\leq 2 (q-1)\int_{\H}\frac{|g|^2}{\|g\|_{L^2(\H)}^2} \log \left(\frac{|g|}{\|g\|_{L^2(\H)}} \right)\,dx\,.
    \end{equation}
    
    A combination of \eqref{Poin.haar.1} together with the horizontal log-Sobolev inequality \eqref{LogSobolevint} proves \eqref{Poi.haar.a}. For the proof of \eqref{Poi.haar.b} we combine inequality \eqref{Poin.haar.1} and \eqref{LogSobolev1}, while the proof of \eqref{Poi.haar.c} is a combination of \eqref{Poin.haar.1} and \eqref{LogSobolev2}. This concludes the proof of the theorem.
\end{proof}

%\begin{cor}
 % The following Poincar\'e type inequalities with respect to the Haar measure on $\H$ hold true
  %\begin{equation}
   %   \label{Poi.haar.a2}
    %   \int_{\H}\frac{|g|^2}{\|g\|_{L^2(\H)}^2}\,dx-\left( \int_{\H}\frac{|g|}{\|g\|_{L^2(\H)}} \,dx\right)^2 \leq \frac{Q}{2} \log \left( C_{B,1} \frac{\|\nabla_{\H}g\|_{L^2(\H)}^{2}}{\|g\|^{2}_{L^2(\H)}}\right)\,, 
  %\end{equation}
  %where the constant $C_{B,1}$ is given in \eqref{bestintsob},
 % \begin{equation}
  %    \label{Poi.haar.b2}
 %\int_{\H}|\frac{|g|^2}{\|g\|_{L^2(\H)}^2}\,dx-\left( \int_{\H}\frac{|g|}{\|g\|_{L^2(\H)}} \,dx\right)^2 \leq \frac{Q}{2s} \log \left( C_{B,s} \frac{\langle \mathcal{L}_sg,g \rangle_{L^2(\H)}}{\|g\|_{L^2(\H)}^{2}}\right)\,,      
  %\end{equation}
  %where the constant $C_{B,s}$ is given in \eqref{bestfrac}, and
  %\begin{equation}
   %   \label{Poi.haar.c2}
    %   \int_{\H}\frac{|g|^2}{\|g\|_{L^2(\H)}^2}\,dx-\left( \int_{\H}\frac{|g|}{\|g\|_{L^2(\H)}} \,dx\right)^2 \leq \frac{Q}{2s} \log \left( C_{B,s}\|U_s\|_{\text{op}} \frac{\langle \mathcal{L}^s g,g \rangle_{L^2(\H)}}{\|g\|_{L^2(\H)}^{2}}\right)\,,
  %\end{equation}
  %where the operator norm $\|U_s\|_{\text{op}}$ has been estimated in \eqref{est1}.
%\end{cor}
The following result is an immediate consequence of Theorem \ref{thm.poi.haar} since for all $x>0$ we have $\log x \leq x-1$.
\begin{cor}
     For $q > 1$ the following Poincar\'e type inequality with respect to the Haar measure on $\H$ holds true
    \begin{equation}
        \label{Poi.haar.a3}
         \int_{\H}|g|^2\,dx-\left( \int_{\H}|g|^{\frac{2}{q}}\,dx\right)^q \leq C_{B,1}\frac{Q(q-1)}{2}\int_{\H}|\nabla_{\H}g|^2\,dx\,,   
    \end{equation}
    where $C_{B,1}$ is given in \eqref{bestintsob}.
    Additionally, for $s \in (0,1]$ and for $q \geq 1$, we have the following Poincar\'e type (modified) fractional inequality  
    \begin{equation}
        \label{Poi.haar.b3}
         \int_{\H}|g|^2\,dx-\left( \int_{\H}|g|^{\frac{2}{q}}\,dx\right)^q \leq C_{B,s}\frac{Q(q-1)}{2s}  \langle \mathcal{L}_sg,g \rangle_{L^2(\H)}\,,
    \end{equation}
    where the constant $C_{B,s}$ is as in \eqref{bestfrac}, as well as the following Poincar\'e type inequality that involves fractional powers of the sub-Laplacian $\mathcal{L}$
     \begin{equation}
        \label{Poi.haar.c3}
         \int_{\H}|g|^2\,dx-\left( \int_{\H}|g|^{\frac{2}{q}}\,dx\right)^q\leq C_{B,s} \|U_s\|_{\text{op}} \frac{Q(q-1)}{2s} \langle \mathcal{L}^s g,g \rangle_{L^2(\H)}\,,
    \end{equation}where the operator norm $\|U_s\|_{\text{op}}$ has been estimated in \eqref{est1}.
\end{cor}
\begin{rem}
    Let us point out that the inequalities \eqref{Poi.haar.a3}, \eqref{Poi.haar.b3} and \eqref{Poi.haar.c3} hold true also when the integration is taken on a ball $B_r(y)$ centered at $y$ with respect to a homogeneous norm on $\H$. Indeed, can repeat the previous argument for the function $b=b(q):= q \log \left(\int_{B_r(y)}|g|^{\frac{2}{q}} \,dx\right)$, for $q >1$, and get 
     \begin{equation*}
           \int_{B_r(y)}|g|^2\,dx-\left( \int_{B_r(y)}|g|^{\frac{2}{q}}\,dx\right)^q \leq 2(q-1) \int_{B_r(y)}|g|^2 \log |g|\,dx\,.
    \end{equation*} The latter inequality for $\frac{g}{\|g\|_{L^2(B_r(y))}}$, combined with the the horizontal log-Sobolev inequality \eqref{LogSobolevint} that holds true on any open set in $\H$, and the properties of the logarithm give
    \begin{equation}\label{Jer.comp.1}
     \int_{B_r(y)}|g|^2\,dx-\left( \int_{B_r(y)}|g|^{\frac{2}{q}}\,dx\right)^q \leq C_{B,1} \frac{Q(q-1)}{2} \int_{B_r(y)} |\nabla_{\H}g|^2\,dx\,.
    \end{equation}
    Moreover, since the inequalities \eqref{LogSobolev1} and \eqref{LogSobolev2} also hold true on any open set in $\H$, we also have 
    \begin{equation}\label{Jer.comp.2}
         \int_{B_r(y)}|g|^2\,dx-\left( \int_{B_r(y)}|g|^{\frac{2}{q}}\,dx\right)^q \leq C_{B,s}\frac{Q(q-1)}{2s}  \langle \mathcal{L}_sg,g \rangle_{L^2(B_r(y))}\,,
    \end{equation}
   and
     \begin{equation}\label{Jer.comp.3}
         \int_{B_r(y)}|g|^2\,dx-\left( \int_{B_r(y)}|g|^{\frac{2}{q}}\,dx\right)^q \leq C_{B,s} \|U_s\|_{\text{op}} \frac{Q(q-1)}{2s} \langle \mathcal{L}^s g,g \rangle_{L^2(B_r(y))}\,.
    \end{equation}
\end{rem}
Note that all the above inequalities are proving an upper bound for the quantity 
\[
  \int_{B_r(y)}|g|^2\,dx-\left( \int_{B_r(y)}|g|^{\frac{2}{q}}\,dx\right)^q\,,
\]
which can be compared with the quantity that appears on the left hand side of the famous Poincar\'e inequality on the balls due to Jerison. Let us recall the latter result:

In 1986 Jerison  \cite{Jer86} in his seminal paper proved the following Poincar\'e inequality on any nilpotent Lie group with respect to the Haar measure. In particular for $B_r(y)$ the ball with respect to the Carnot-Carath\'eodory distance centered at $y$ and of radius $r$, Jerison showed the following: 
	
\begin{thm}
	\label{thm.jerison}
For any $p \in [1,\infty)$, there exists a constant $P_{0}(r)=P_{0}(r,p)$ such that for all $f \in C^{\infty}(B_r(x))$
	\begin{equation}\label{jer}
	\int_{B_{r}(y)} |f(x)-f_{B_{r}(y)}|^{p}\,dx \leq P_{0}(r) \int_{B_{r}(y)} |\nabla_{\mathbb G}f(x)|^p\,dx\, ,
	\end{equation}
	where $f_{B_{r}(y)}:= \frac{1}{|B_{r}(y)|}\int_{B_{r}(y)}f(x)\,dx$.
\end{thm} 
In the next theorem, we generalise Jerison's result in the case of $\H$ and when $p=2$, by allowing the ball $B_r(y)$ to be regarded with  respect to any homogeneous quasi-norm on $\H$ and the right-hand side of \eqref{jer} to be bounded  by the infimum of certain quantities that involve the (modified) fractional sub-Laplacian $\mathcal{L}$.
\begin{cor}\label{cor.jer}
Let $B_r(y)$ be the ball of radius $r$ centered at $y \in \H$ and with respect to some homogeneous quasi-norm on $\H$ such that $\text{Vol}(B_r(y)) \leq 1$. Then, we have
\begin{equation}\label{upper.b.j}
\int_{B_{r}(y)} |g(x)-g_{B_{r}(y)}|^{2}\,dx \leq M
\end{equation}
where 
\[
M=\inf_{s \in(0,1]}C_{B,s}\frac{Q}{2s}  \langle \mathcal{L}_sg,g \rangle_{L^2(B_r(y)}\,,
\]
where the constant $C_{B,s}$ is given in \eqref{bestfrac}, respectively. In particular, we have 
\[
\int_{B_{r}(y)} |g(x)-g_{B_{r}(y)}|^{2}\,dx \leq \frac{(n!)^{\frac{1}{n+1}}}{\pi n^{2}}\frac{Q}{2}\int_{B_r(y)}|\nabla_{\H} g(x)|^2\,dx\,,
\]
since $C_{B,1}=\frac{(n!)^{\frac{1}{n+1}}}{\pi n^{2}}$.
\end{cor}
\begin{rem}
    Note that the constant $C_{B,s}$ depends only on the dimension of $\H$; i.e., it can be regarded as fixed. Hence the constant that appears in the upper bound for the quantity \eqref{upper.b.j} is independent of any involved parameters. 
\end{rem}
\begin{proof}[Proof of Corollary \ref{cor.jer}]
 Observe that the integral $I=\int_{B_{r}(y)} |g-g_{B_{r}(y)}|^{2}\,dx$ can be rewritten as 
   \begin{eqnarray}
\label{thmJer.b}
 	I & = &  \int_{B_r(y)} \left[g^2-\frac{2g}{\text{Vol}(B_r(y))} \int_{B_r(y)}g\,dx+\frac{1}{\text{Vol}^{2}(B_r(y))}\left(  \int_{B_r(y)}g\,dx \right)^2\right]\,dx \nonumber\\
 & = &  \int_{B_r(y)} g^2\,dx-\frac{2}{\text{Vol}(B_r(y))} \left(  \int_{B_r(y)}gdx\right)^2+ \frac{1}{\text{Vol}(B_r(y))} \left(  \int_{B_r(y)}g\,dx\right)^2 \nonumber\\
 & = &  \left[\int_{B_r(y)} g^2\,dx-\frac{1}{\text{Vol}(B_r(y))}\left(  \int_{B_r(y)}g\,dx\right)^2\right]\,.
   \end{eqnarray}
   Now, since $\text{Vol}(B_r(y)) \leq 1$,
   we get 
   \[
  \int_{B_{r}(y)} |g(x)-g_{B_{r}(y)}|^{2}\,dx\leq  \int_{B_r(y)}|g|^2\,dx-\left( \int_{B_r(y)}|g|\,dx\right)^2\,.
   \]   
   Combining the latter with \eqref{Jer.comp.2} for $q=2$ gives the desired result, and the proof is complete. 
\begin{rem}
Let us note that the results of these sections can be generalised in the setting of other stratified groups, or even graded groups, whenever the latter makes sense. However, in this case, we will only have representations of the involved constants in terms of ground state solutions to relevant nonlinear PDEs. Precisely, Theorem \ref{thm.poi.H} is true for all stratified groups due to Theorem 7.2 in \cite{CKR21c}, and Theorem 8.2 can be stated for all graded Lie groups thanks to Theorem 4.3 in \cite{CKR21c}. The details of these results will appear elsewhere.
\end{rem}
    
\end{proof}

\end{document}